\documentclass[a4paper,10.9pt]{amsart}

\DeclareRobustCommand{\SkipTocEntry}[5]{}
\usepackage{accents}

\usepackage{marvosym}

\usepackage[T1]{fontenc}
\usepackage{empheq}
\usepackage{relsize}
\usepackage{amsmath}

\usepackage{bm}

\usepackage{upgreek}
\usepackage{ esint }
\usepackage{color}
\usepackage{comment}
\usepackage{amssymb}
\usepackage{amsfonts}
\usepackage{graphicx}

\usepackage{amscd}

\usepackage{slashed}
\usepackage{pdflscape}
\usepackage{tikz}
\usepackage{enumerate}
\usepackage{multirow}
\usepackage{stmaryrd}

\usepackage{scalerel,stackengine}

\usepackage{ifthen}

\usepackage{bbold}
\usepackage[linkcolor=blue]{hyperref}
\usepackage{mathrsfs}
\usepackage[colorinlistoftodos]{todonotes}

\definecolor{blue}{rgb}{.255,.41,.884} 
\definecolor{red}{rgb}{1, 0, 0} 
\definecolor{green}{rgb}{.196,.804,.196} 
\definecolor{yellow}{rgb}{1,.648,0} 
\definecolor{pink}{rgb}{1,0.5,0.5}

\newtheorem{theorem}{Theorem}[section]
\newtheorem{lemma}[theorem]{Lemma}
  
\newtheorem{proposition}[theorem]{Proposition}
\newtheorem{corollary}[theorem]{Corollary}

\theoremstyle{definition}
\newtheorem{definition}[theorem]{Definition}
\newtheorem{example}[theorem]{Example}

\theoremstyle{remark}
\newtheorem{remark}[theorem]{Remark}

\newtheorem{problem}[theorem]{Problem}

\usepackage{pdfsync}

\usepackage{graphicx} 
\usepackage{amsmath} 
\usepackage{amsfonts}
\usepackage{amssymb}

\newcommand{\be}{\begin{equation}}

\newcommand{\ee}{\end{equation}}

\newcommand{\II}{{\rm  I\hspace{-.2mm}I}}
\newcommand{\IIo}{\hspace{0.4mm}\mathring{\rm{ I\hspace{-.2mm} I}}{\hspace{.0mm}}}
\newcommand{\iio}{\hspace{0.4mm}\mathring{{ \iota\hspace{-.41mm} \iota}}{\hspace{.0mm}}}
\newcommand{\Fo}{ {\hspace{.6mm}} \mathring{\!{ F}}{\hspace{.2mm}}}
\newcommand{\IIIo}{{\mathring{{\bf\rm I\hspace{-.2mm} I \hspace{-.2mm} I}}{\hspace{.2mm}}}{}}
\newcommand{\IVo}{{\mathring{{\bf\rm I\hspace{-.2mm} V}}{\hspace{.2mm}}}{}}

\newcommand{\ba}{\begin{array}}

\newcommand{\ea}{\end{array}}

\newcommand{\beq}{\begin{eqnarray}}

\newcommand{\eeq}{\end{eqnarray}}

\newtheorem{lm}{lemma}

\newtheorem{thee}{theorem}

\newtheorem{proo}{proposition}

\newtheorem{co}{corollary}

\newtheorem{rem}{remark}

\newtheorem{deff}{definition}

\newcommand{\bd}{\begin{deff}}

\newcommand{\ed}{\end{deff}}

\newcommand{\bl}{\begin{lm}}

\newcommand{\el}{\end{lm}}

\newcommand{\bp}{\begin{proo}}

\newcommand{\ep}{\end{proo}}

\newcommand{\bt}{\begin{thee}}

\newcommand{\et}{\end{thee}}

\newcommand{\bc}{\begin{co}}

\newcommand{\ec}{\end{co}}

\newcommand{\brm}{\begin{rem}}

\newcommand{\erm}{\end{rem}}

\usepackage{t1enc}
\def\frak{\mathfrak}

\def\Cal{\mathcal}

\newcommand{\bS}{\mathbb{S}}

\newcommand{\newc}{\newcommand}

\let\ccdot.

\newc{\aR}{\mbox{\boldmath{$ R$}}}
\newc{\aS}{\mbox{\boldmath{$ S$}}}
\newc{\aT}{\mbox{\boldmath{$ T$}}}
\newc{\aW}{\mbox{\boldmath{$ W$}}}

\newc{\aD}{\mbox{\boldmath{$ D$}}\hspace{-.2mm}}

\newc{\aK}{\mbox{\boldmath{$ K$}}}
\newc{\aL}{\mbox{\boldmath{$ L$}}}

\newcommand{\ce}{{\Cal E}}

\newcommand{\ct}{{\Cal T}}

\newcommand{\bT}{{\Bbb T}}

\usepackage{amssymb}
\usepackage{amscd}

\newcommand{\Rho}{{\it P}}

\let\hash=\sharp  

\let\i=\iota

\newcommand{\nn}[1]{(\ref{#1})}

\newcommand{\bg}{\mbox{\boldmath{$ g$}}}

\newc{\obstrn}[2]{B^{#1}_{#2}}

\newcommand{\rpl}                         
{\mbox{$
\begin{picture}(12.7,8)(-.5,-1)
\put(0,0.2){$+$}
\put(4.2,2.8){\oval(8,8)[r]}
\end{picture}$}}

\newcommand{\lpl}                         
{\mbox{$
\begin{picture}(12.7,8)(-.5,-1)
\put(2,0.2){$+$}
\put(6.2,2.8){\oval(8,8)[l]}
\end{picture}$}}

\newc{\tensor}[1]{#1}
\newc{\Mvariable}[1]{\mbox{#1}}
\newc{\down}[1]{{}_{#1}}
\newc{\up}[1]{{}^{#1}}

\newc{\JulyStrut}{\rule{0mm}{6mm}}
\newc{\midtenPan}{\mbox{\sf S}}
\newc{\midten}{\mbox{\sf T}}
\newc{\midtenEi}{\mbox{\sf U}}
\newc{\ATen}{\mbox{\sf E}}
\newc{\BTen}{\mbox{\sf F}}
\newc{\CTen}{\mbox{\sf G}}

\def\sideremark#1{\ifvmode\leavevmode\fi\vadjust{\vbox to0pt{\vss
 \hbox to 0pt{\hskip\hsize\hskip1em
 \vbox{\hsize2cm\tiny\raggedright\pretolerance10000
  \noindent #1\hfill}\hss}\vbox to8pt{\vfil}\vss}}}

\numberwithin{equation}{section}

\newcommand{\hh}{{\hspace{.3mm}}}

\newcommand{\cc}{\boldsymbol{c}}

\newcommand{\pdot}{{\boldsymbol{\cdot}}}

\newcommand{\bn}{\bar{\nabla}}

\newcommand{\sss}{\scriptscriptstyle}

\renewcommand\geq{\geqslant}
\renewcommand\leq{\leqslant}

\DeclareMathOperator{\EXT}{d}
\newcommand{\ext}{{\EXT\hspace{.01mm}}}

\DeclareMathOperator{\Vol}{Vol}

\stackMath
\newcommand\reallywidehat[1]{%
\savestack{\tmpbox}{\stretchto{%
  \scaleto{%
    \scalerel*[\widthof{\ensuremath{#1}}]{\kern-.6pt\bigwedge\kern-.6pt}%
    {\rule[-\textheight/2]{1ex}{\textheight}}
  }{\textheight}%
}{0.5ex}}%
\stackon[1pt]{#1}{\tmpbox}%
}

\begin{document}


\renewcommand{\today}{}
\title{
{Generalized Willmore Energies, $Q$-curvatures,
Extrinsic Paneitz Operators, and Extrinsic Laplacian Powers
}}
%
%
%

\author{ Samuel Blitz${}^\flat$, A. Rod Gover${}^\sharp$ \&  Andrew Waldron${}^\natural$}

\address{${}^\flat$
  Center for Quantum Mathematics and Physics (QMAP)\\
  Department  of Physics\\ 
  University of California\\
  Davis, CA95616, USA} 
   \email{shblitz@ucdavis.edu}
 
\address{${}^\sharp$
  Department of Mathematics\\
  The University of Auckland\\
  Private Bag 92019\\
  Auckland 1142\\
  New Zealand,  and\\
  Mathematical Sciences Institute, Australian National University, ACT 
  0200, Australia} \email{gover@math.auckland.ac.nz}
  
  \address{${}^{\natural}$
  Center for Quantum Mathematics and Physics (QMAP)\\
  Department of Mathematics\\ 
  University of California\\
  Davis, CA95616, USA} \email{wally@math.ucdavis.edu}

\vspace{10pt}

\renewcommand{\arraystretch}{1}

\begin{abstract}
Over forty years ago, Paneitz, and independently Fradkin and Tseytlin, discovered a fourth-order conformally-invariant  differential operator, intrinsically defined on a conformal manifold, mapping scalars to scalars.
This operator is a special case of the so-termed {\it extrinsic Paneitz operator} defined in the case when the conformal manifold is itself a conformally embedded hypersurface. In particular, this
encodes the obstruction to smoothly solving the
five-dimensional scalar Laplace equation, and suitable higher dimensional analogs, on conformally compact structures with constant scalar curvature.
Moreover, the extrinsic Paneitz operator can act on tensors  
of general type by dint of being defined on tractor bundles.
Motivated by a host of applications, 
we explicitly compute the extrinsic Paneitz operator. We apply this formula to obtain:
an extrinsically-coupled~$Q$-curvature for embedded four-manifolds, the anomaly in renormalized volumes for conformally compact five-manifolds with negative constant scalar curvature, Willmore energies for embedded four-manifolds, the local obstruction to smoothly solving the five-dimensional singular Yamabe problem, and   new extrinsically-coupled  fourth and sixth order operators for embedded surfaces and four-manifolds, respectively.

\vspace{10cm}


\end{abstract}


\maketitle

\pagestyle{myheadings} \markboth{Blitz, Gover, \& Waldron}{Extrinsic Paneitz Operators}

\newpage

\tableofcontents

\newcommand{\balpha}{{\bm \alpha}}
\newcommand{\balphas}{{\scalebox{.76}{${\bm \alpha}$}}}
\newcommand{\bnu}{{\bm \nu}}
\newcommand{\blambda}{{\bm \lambda}}
\newcommand{\bnus}{{\scalebox{.76}{${\bm \nu}$}}}
\newcommand{\bnuss}{\hh\hh\!{\scalebox{.56}{${\bm \nu}$}}}

\newcommand{\bmu}{{\bm \mu}}
\newcommand{\bmus}{{\scalebox{.76}{${\bm \mu}$}}}
\newcommand{\bmuss}{\hh\hh\!{\scalebox{.56}{${\bm \mu}$}}}

\newcommand{\btau}{{\bm \tau}}
\newcommand{\btaus}{{\scalebox{.76}{${\bm \tau}$}}}
\newcommand{\btauss}{\hh\hh\!{\scalebox{.56}{${\bm \tau}$}}}

\newcommand{\bsigma}{{\bm \sigma}}
\newcommand{\bsigmas}{{{\scalebox{.8}{${\bm \sigma}$}}}}
\newcommand{\bbeta}{{\bm \beta}}
\newcommand{\bbetas}{{\scalebox{.65}{${\bm \beta}$}}}

\renewcommand{\bS}{{\bm {\mathcal S}}}
\newcommand{\bB}{{\bm {\mathcal B}}}
\renewcommand{\bT}{{\bm {\mathcal T}}}
\newcommand{\bM}{{\bm {\mathcal M}}}

\newcommand{\go}{{\mathring{g}}}
\newcommand{\nuo}{{\mathring{\nu}}}
\newcommand{\alphao}{{\mathring{\alpha}}}

\newcommand{\Ell}{\mathscr{L}}
\newcommand{\density}[1]{[g\, ;\, #1]}

\renewcommand{\Dot}{{\scalebox{2}{$\cdot$}}}

\newcommand{\PanE}{P_{4}^{\sss\Sigma\hookrightarrow M}}
\newcommand\eqSig{ \mathrel{\overset{\makebox[0pt]{\mbox{\normalfont\tiny\sffamily~$\Sigma$}}}{=}} }
\renewcommand\eqSig{\mathrel{\stackrel{\Sigma\hh}{=}} }
\newcommand\eqtau{\mathrel{\overset{\makebox[0pt]{\mbox{\normalfont\tiny\sffamily~$\tau$}}}{=}}}
\newcommand{\hd }{\hat{D}}
\newcommand{\hdb}{\hat{\bar{D}}}
\newcommand{\Two}{{{{\bf\rm I\hspace{-.2mm} I}}{\hspace{.2mm}}}{}}
\newcommand{\TwoN}{{\mathring{{\bf\rm I\hspace{-.2mm} I}}{\hspace{.2mm}}}{}}
\newcommand{\Fn}{\mathring{\mathcal{F}}}
\newcommand{\csdot}{\hspace{-0.75mm} \cdot \hspace{-0.75mm}}
\newcommand{\IdD}{(I \csdot \hd)}
\newcommand{\Kd}{\dot{K}}
\newcommand{\Kdd}{\ddot{K}}
\newcommand{\Kddd}{\dddot{K}}

 \newcommand{\bdot }{\mathop{\lower0.33ex\hbox{\LARGE$\cdot$}}}

\definecolor{ao}{rgb}{0.0,0.0,1.0}
\definecolor{forest}{rgb}{0.0,0.3,0.0}
\definecolor{red}{rgb}{0.8, 0.0, 0.0}


\section{Introduction}


The Paneitz operator is the fourth-order, conformally covariant, squared-Laplacian
operator on functions 
of an $n\geq 3$ dimensional Riemannian manifold,
\begin{equation} \label{P4-intr}
 P_4  :=
 \Delta^2 + \nabla_a \circ (4 P^{ab} - (n-2) J g^{ab}) \circ \nabla_b + \tfrac{n-4}{2} \hh {\mathcal Q}_{n}(g)\, , 
 \end{equation}
where (see Section~\ref{stRm} for standard tensor definitions and conventions)
$${\mathcal Q}_{n}(g) := -\Delta J-2P^2
+  \tfrac{n}{2} J^2  \, .
$$
It was discovered in independent contexts by Paneitz~\cite{Paneitz} as well as Fradkin and Tseytlin~\cite{Fradkin}. Its
perceived importance was elevated significantly by its connection to
the conformal transformation of the local Riemannian 4-manifold invariant known as
(Branson's)~$Q$-curvature $Q_4(g):={\mathcal Q}_{4}(g)$. If~$\hat{g}$ and~$g$ are metrics conformally related by
$\hat{g}=e^{2\omega}g$, for some smooth function~$\omega$, then
\begin{equation}\label{Qt}
e^{4\omega} Q_4(\hat g) = Q_4(g)+P_4\hh \omega  .
\end{equation}
In dimension four,~$P_4\hh\omega$ is a divergence while the factor~$e^{4\omega}$ exactly balances the change in the metric measure arising from the conformal change,  and so~$\int \! Q$ is a global
conformal invariant of closed Riemanian four-manifolds.  This curvature was
introduced in \cite{BO} to help understand  Polyakov-type formul\ae~\cite{Poly1}
for the conformal variant of functional determinants. It and its
higher dimensional analogs were subsequently found to also have 
applications in, for example, geometric scattering~\cite{GZscatt,StChang} and curvature prescription
problems, see~\cite{Origins,WhatQ,ChYa} for reviews and further references.

\smallskip

Perhaps the greatest importance of the Paneitz operator, the
$Q$-curvature, and their higher order analogs is through their links
to (and origins in \cite{BransonSharp,GJMS}) a class of problems that
fall under the umbrella of {\em geometric holography}. For example, for
the~$Q$-curvature on a closed 4-manifold~$(\Sigma,g_\Sigma)$, the integral~$\int_\Sigma
Q$ gives the log coefficient (the so-called anomaly term) in the
volume asymptotics of the associated Poincar\'{e}-Einstein five-manifold
that has~$\Sigma$ as the boundary at infinity~\cite{HS,HS1,GrVol}. On the
other hand the functional gradient of this integral, with respect to
metric variations, is the Bach tensor (which orginated in conformal
gravity~\cite{Bach}). The Bach tensor also arises as the obstruction
to {\em smoothly} solving for a five-dimensional  Poincar\'e-Einstein metric~\cite{FGbook}, given some boundary data~$(\Sigma,g_\Sigma)$. So overall there is a
rather beautiful and powerful picture.

From a geometric holography viewpoint this picture has a natural
and rich extension which  motivates the current
work. 
This begins with dropping the restriction to
Poincar\'{e}-Einstein manifolds. First recall that a  conformal class $\cc$ 
is an equivalence class of metrics where $\hat g \sim g\in \cc$ when $\hat g=\Omega^2 g$ for some $0<\Omega\in C^\infty M$. Then, 
given a general conformal manifold $(M,\cc)$,
with boundary $\Sigma$,
one might hope, in the spirit of the Poincar\'e--Einstein
program just mentioned, to study the conformal geometry of the
boundary {\em including its extrinsic data} $\Sigma\hookrightarrow (M,\cc)$ by finding, from the
conformal class $\cc$ restricted to  the manifold
interior, a canonically determined metric~$g^o$. 
In fact, we are interested in a  version of this set-up for general hypersurfaces embedded in any conformal manifold $(M,\cc)$, where one then studies metrics $g^o$ on $M\backslash \Sigma$ for which $\Sigma$ is a conformal infinity.
Typically, there is no Einstein metric in the conformal
class~$\cc$. 
When~$\cc$ does contain an
Einstein metric with conformal infinity~$\Sigma$, the hypersurface embedding 
 is necessarily totally umbilic~\cite{LeBrun,Goal} and satisfies further
extremely restrictive related higher order conditions~\cite{BGW}. Thus instead it is natural to
seek, from the conformal class on $M\backslash \Sigma$, a metric that has
constant scalar curvature and $\Sigma$ as a conformal infinity.
Since Einstein metrics have constant scalar
curvature by dint of the Bianchi identities, this is strictly a
weakening of the Poincar\'e--Einstein condition. By rough counting this
{\it singular Yamabe problem} also has the right degrees of freedom---one positive function (the rescaling of the metric $g^o$) fixes one
scalar curvature quantity determined by the metric. The key question then is whether such metrics are uniquely determined by the  embedding data~$\Sigma\hookrightarrow (M,\cc)$. A result of Andersson, Chru\'sciel and Friedrich~\cite{ACF} shows that, at least to a certain order,  asymptotics  of~$g^o$ are uniquely and locally determined by the conformal embedding; see Theorem~\ref{Sc1} below.

In~\cite{WillB,Will1} it was demonstrated that the Andersson--Chru\'sciel--Friedrich result could be employed to study conformal hypersurface embeddings. In particular, those studies yielded extrin\-sically-coupled analogs of the Paneitz operator and its higher Laplacian power counterparts which also extend to higher order operators on tensors~\cite{GW,Will2}. This extension was achieved by employing the
natural invariant calculus on conformal manifolds
and hypersurfaces known as
  tractor calculus~\cite{Thomas, BEG}. 
Fundamental objects therein are weight $w$  tensor  tractor bundles~$\ct^\Phi M[w]$.
Tractor bundles are explained in  Section~\ref{Bground}.
For any vector bundle ${\mathcal V}M$ over $M$,  we denote the product ${\mathcal V}M\otimes\ce M[w]$ by 
${\mathcal V}M[w]$, where $\ce M[w]$ is  the bundle of weight $w$ conformal densities. A section of $\ce M[w]$ is an   
equivalence class~$[g;\varphi]$ where
$(\hat g,\hat \varphi) \sim (g,\varphi)\in f$ when $\hat g=\Omega^2 g$ and $\hat \varphi=\Omega^w \varphi$ for some $0<\Omega\in C^\infty M$;
again see Section~\ref{Bground}.

Our central Theorem~\ref{Big_Formaggio} gives an explicit formula for the extrinsically-coupled Paneitz operator~$\PanE$ introduced in~\cite{GW,Will1,Will2} (see Section~\ref{prufs}) acting on general tractors:
\begin{equation}\label{magna_carta}
\PanE
:\Gamma\Big(\ct^\Phi M\Big[\tfrac{5-d}2\Big]\Big|_\Sigma\Big)\longrightarrow\Gamma\Big(
\ct^\Phi M\Big[\tfrac{-3-d}2\Big]\Big|_\Sigma\Big)\, .
\end{equation}
This is a fourth order (Laplacian-squared type) operator defined along $\Sigma$, and uniquely determined by the conformal embedding $\Sigma\hookrightarrow (M,\cc)$.
Importantly our formula is strongly invariant, in the sense that the above bundles can be tensored with any vector bundle over $M$ equipped with a connection. Twisting the tractor connection appearing in $\PanE$ with that connection
still yields a conformally invariant operator. The same statement does not apply, for example, to the original Paneitz operator of Equation~\nn{P4-intr}---indeed  this could not be the case without 
violating known non-existence results  for conformally invariant cubic Laplacian powers~\cite{GraNon} discussed in Section~\ref{doIcogito}.
Partly because of its strong invariance, our central  result for the extrinsically-coupled Paneitz operator  underpins a slew of applications, which we now describe.

\smallskip
Specializing to scalars, the extrinsic Paneitz operator can be expressed as a product of four degenerate Laplace-type operators: If $\psi$ is any smooth extension of $\overline \psi\in C^\infty \Sigma$, denoting by~$\Delta^o$ the Laplacian of the metric $g^o=s^{-2}g$
for $s\in C^\infty M$ any solution to conditions {\it (i-iii)} of Theorem~\ref{Sc1}, then  the scalar quantity 
$$
\scalebox{.86}{$
s^{-\frac{3+d}2}
\Big(\Delta^o+\dfrac14
(d-3)(d+1)\Big)
\circ
\Big(\Delta^o+\dfrac14
(d-1)^2\Big)
\circ
\Big(\Delta^o+\dfrac14
(d+1)(d-3)\Big)
\circ
\Big(\Delta^o+\dfrac14
(d+3)(d-5)\Big)\circ \!s^{\frac{d-5}2}\psi\, ,
$}
$$
defined along $M\backslash \Sigma$, continues smoothly to $\Sigma$.
Its extension to $\Sigma$ is a fourth order differential operator that acts
tangentially to the boundary.  The above is a rewriting of the fourth
power of the Laplace--Robin operator; see Remark~\ref{PIB}. The above property
of the Laplace--Robin operator and its {\it r\^ole} played in general
conformal boundary problems are explained in~\cite{GW}.  Indeed its
restriction to $\Sigma$ is independent of the choice of
extension the~$\psi$ of $\bar \psi$ and thus defines $ \PanE \big(\overline\psi \hh\big)$.  Moreover it
only depends on the conformal embedding data and enjoys the conformal
covariance 
$$
\bar\Omega^{\frac{d+3}{2}} \Big(
P_4^{\Sigma \hookrightarrow M}(\Omega^2 g) \Big)\big(\bar\Omega^{\frac{5-d}2}\overline \psi\hh\big)=\Big(P_4^{\Sigma \hookrightarrow M}(g)\Big) \big(\overline\psi\hh\big)\, ,
$$
where $\bar\Omega=\Omega|_\Sigma$.

Another interesting application of the extrinsically coupled Paneitz operator, and its generalization  given in~\cite[Theorem 4.1]{GW}, to {\it conformally compact structures}
$(M,s,g^o)$ (where~$s^2 g^o$ is a Riemannian metric on $M$ and $s$
 obeys Condition~(\ref{cond1}) of Theorem~\ref{Sc1}) is to asymptotic solutions of the five-dimensional Laplace equation
 \begin{equation}\label{siren}
 \Delta^o f = 0 \, .
 \end{equation}
 Here the function $f\in  C^\infty M$ has prescribed Dirichlet data $f|_{\partial M}=:\bar f$. When $g^o$ has constant negative scalar curvature, the formal solution to this problem beyond order $s^3$ in $f$ is obstructed
by the extrinsic Paneitz operator acting on $\bar f$, or more generally its conformally compact analog given in~\cite{GW}. In fact this picture generalizes to Laplace-type equations for  higher tensors and tractors in general dimensions.

\smallskip

Our first appplication of 
Theorem~\ref{Big_Formaggio} expresses the extrinsic Paneitz operator acting on scalars as the sum of the intrinsic Paneitz operator of Equation~\nn{P4-intr} and hypersurface operators:


\begin{corollary}
\label{forIamamerecorollary} 
Let~$d\geq 5$. Acting on conformal densities of weight $\frac{5-d}{2}$ along $\Sigma$, 
the extrinsic Paneitz operator is given by
$$
\PanE  \!= \bar{\Delta}^2 +
\bar{\nabla}^a \!\circ \big(4 \bar{P}_{ab} - (d\hh\!-\hh\!3) \bar{J} \bar{g}_{ab} + 8 \IIo_{ab}^2 + \tfrac{d^2 - 4d-1}{2(d-1)(d-2)} K \bar{g}_{ab} + 4(d\!\hh-\hh\!2) \Fo_{ab} \big) \circ \bar\nabla^b 
- \tfrac{5-d}2 \, {\mathcal Q}^{\Sigma \hookrightarrow M}_{d-1}\!(g)\, , 
$$
where the multiplication operator
\begin{align}
\begin{split}
\!\!\!\!\!\!{\mathcal Q}^{\sss\Sigma \hookrightarrow M}_{d-1}&(g) \,:=\:\!
- \bar{\Delta} \bar{J} -2\bar{P}^2
+  \tfrac{d-1}{2} \bar{J}^2\\
& 
+\! \tfrac{2}{d-1} \IIo \csdot (\bar{\Delta} \IIo) 
\\&
+\!\tfrac{1}{d-4}\Big\{\:\:\:
 2(d-2)\bar{\nabla} \csdot \bar{\nabla} \csdot \Fo
 +\scalebox{.9}{$\tfrac{3d^2 - 9d + 4}{2(d-1)(d-2)}$} \hh\bar{\Delta} K 
+ 4\bn^a (\IIo_a \csdot \bn \csdot \IIo)\\&
\phantom{oppera}
-\!2(d-2)\IIo \csdot C_{\hat n}^\top 
- \tfrac{6(d-2)}{d-1} \IIo^{ab} \bn^c W_{cab\hat{n}}^\top \\[1mm]
&
\phantom{oppera}
- \!2(d-2)(d-5) \Fo \csdot \bar{P} 
- 4(d-6) \IIo \csdot \bar{P} \csdot \IIo 
- \tfrac{d^3 + 2d^2 - 27d+44}{2(d-1)(d-2)} \bar{J} K  \\[1mm]
&
\phantom{oppera}
+ \!2(d-2)(d-3) H \IIo \csdot \Fo
- 2(d-2) H \IIo^3 \\
&
\phantom{oppera}
+\! \tfrac{2(d+2)}{d-1} \IIo^{ad}\IIo^{bc}  \bar{W}_{abcd} 
\!+\! 2(d\!-\!2)\big((d\!-\!2)\Fo^2\!+\!
\IIo \csdot \Fo \csdot \IIo\big)
\!+\! \scalebox{.9}{$\tfrac{17d^3 -86d^2 + 133d - 52}{8(d-1)(d-2)^2 }$} K^2 \Big\}
 \, .
\end{split} \label{ext-Q}
\end{align}
\end{corollary}
\noindent
The above result has been verified using an independent method in~\cite[Theorem 1]{Jagain}.
  The extrinsic curvatures appearing in the above corollary are defined in Section~\ref{conbed}. For now note that $\IIo$
   is the trace-free second fundamental form  while $\mathring F$
    is the trace-free part of the so-called {\it Fialkow tensor}. The latter is an example of a third fundamental form, in the sense defined in~\cite{BGW}; again see Section~\ref{conbed}. The weight $-2$ conformal density $K$ is the square of the trace-free second fundamental form and is termed the {\it rigidity density} because it was employed in~\cite{Polyakov} to describe rigid fundamental strings. The mean curvature of $\Sigma\hookrightarrow (M,g)$ is denoted by~$H$.
Also~$\bar \nabla$ is the Levi-Civita connection of the metric $\bar g$ induced along $\Sigma$ by $g\in \cc$ and $\bar \Delta$ the corresponding  Laplacian. The Schouten tensor of this metric is denoted  $\bar \Rho$ while $\bar J$ is its trace. Also,
a relatively simple computation shows that residue of the $d=4$ pole in Equation~\nn{ext-Q} reproduces the functional gradient 
of the integrated three-manifold extrinsic  
$Q$-curvature computed in~\cite{hypersurface_old}.
This is an odd dimensional boundary  example of a conjecture made in~\cite[Remark 5.5]{Will2} and verified for two dimensional boundaries.
A proof for all even boundary dimensions may be found in~\cite{Jodd}.

Just as the intrinsic Paneitz operator is linked to a $Q$-curvature,  the same holds for the extrinsic Paneitz operator.
The four-manifold {\it extrinsic $Q$-curvature} is defined by extending the extrinsic Paneitz operator to act on log densities~\cite{GW,BGW}  (see also~\cite{FEFF-HIR} for an ``ambient metric'' approach to the intrinsic case), 
\begin{equation}\label{creatine}
Q_4^{\Sigma \hookrightarrow M}(g_\tau):=\PanE\log \tfrac1\tau\big|_{\Sigma,g_\tau}\, .
\end{equation} 
Note that the choice of a positive conformal density $\tau=[g;t]=[g_\tau;1]$ is completely equivalent to that of a Riemannian metric $g_\tau\in \cc$.
The right hand side of the above display is a weight $-4$ density evaluated along the boundary in the scale $g_\tau$.
Also,  log densities are explained in Section~\ref{Bground}.
By construction,
the extrinsic $Q$-curvature obeys the linear shift property
\begin{equation}\label{linearshift}
Q_4^{\Sigma\hookrightarrow M}(e^{2\omega} g)=e^{-4\omega} \big[
Q_4^{\Sigma\hookrightarrow M}(g)
+P^{\Sigma\hookrightarrow M} \bar\omega\big]\, ,
\end{equation}
where $\omega \in C^\infty M$ and $\bar \omega=\omega|_\Sigma$.
In the above we used that the extrinsic Paneitz operator for embedded four-manifolds acts on weight zero densities, and that these are equivalent to smooth functions.
In particular, for $\bar f\in C^\infty \Sigma$, we have
\begin{align}\label{Pfab}
\PanE \bar f = \bar{\Delta}^2\bar  f +
\bar{\nabla}^a \circ \big(4 \bar{P}_{ab} - 2 \bar{J}\hh  \bar{g}_{ab} + 8 \IIo_{ab}^2  + 12 \Fo_{ab} + \tfrac{1}{6} K \bar{g}_{ab}\big) \circ \bar\nabla^b \bar f \, .
\end{align}
The first three terms displayed on the right hand side are precisely the Paneitz operator intrinsic to the conformal 4-manifold~$(\Sigma,\cc_\Sigma)$. 
Just as for its intrinsic counterpart, the extrinsic~$Q$-curvature  can also be determined from ${\mathcal Q}_4^{\Sigma \hookrightarrow M}$ in Corollary~\ref{forIamamerecorollary} (see Lemma~\ref{fiatlux}); this is recorded in the following theorem. 
\begin{theorem}\label{BQC}
Let~$\Sigma\hookrightarrow (M,\cc_\Sigma)$ be a conformally embedded hypersurface.
Then, given~$g\in \cc$ the extrinsically-coupled~$Q$-curvature is 
\begin{align*}
Q_4^{\Sigma\hookrightarrow M}( g) = Q^\Sigma_4 + W\! m + U+Qe\, ,
\end{align*}
where
\begin{align*}
Q^\Sigma_4 \hh &\, =
-\bar{\Delta} \bar{J}
- 2 \bar{P}^2
+2\bar{J}^2   \, ,\\
 W\! m  &:=\:
\phantom{+}\tfrac12 \IIo \csdot \bar{\Delta} \IIo 
+\tfrac43\bar{\nabla}^a \big(\IIo_{a} \csdot \bar{\nabla} \csdot \IIo \big)
+\tfrac32 \bar \Delta K\\  
&\phantom{=}\:\:\: 
  -6 \IIo \csdot C_n^{\top}
  + 4 \IIo \csdot \bar{P} \csdot \IIo
  - \tfrac72 \bar{J} K 
 -6 H \IIo^3
  + 12 H \IIo \csdot \Fo  \, ,
  \\[1mm] 
U\:\: &:=
 \phantom{+}18 \Fo \csdot \Fo
+6 \IIo \csdot \Fo \csdot \IIo
+ \tfrac{49}{24} K^2  
 -\tfrac92 \IIo^{ab} \bar{\nabla}^c W_{cab \hat{n}}^\top 
+ \tfrac72 \IIo^{ad} \IIo^{bc} \bar{W}_{abcd} \color{black}\, ,\\[1mm] 
 Qe\:\hh
 &:=\:
 \phantom{+}\tfrac83\bar{\nabla}^a \big(\IIo_{a} \csdot \bar{\nabla} \csdot \IIo \big)
 +6 \bar{\nabla} \csdot \bar{\nabla} \csdot \Fo 
 - \tfrac{1}{12} \bar{\Delta} K  \, ,
\end{align*}
and~$W\! m$ and~$U$ 
are conformally invariant   weight~$-4$ conformal densities. 
\end{theorem}

This theorem leads to a formula for the anomaly in the renormalized volume of the manifold~$(M\backslash\Sigma,s^{-2}g)$  as described in Section~\ref{Suitable}. A second application is to higher Willmore energies for conformally embedded hypersurfaces. 
These applications both rely on the fact that the linear shift Property~\nn{linearshift} 
together with Equation~\nn{Pfab}
imply that the integral 
\begin{equation}\label{WE1}
\int_\Sigma \ext \! \Vol(\bar g)
\hh Q^{\Sigma\hookrightarrow M}(g)
\end{equation}
is independent of the choice of $g\in \cc$. Note that here
$\bar g$ is the metric induced along $\Sigma$ by $g$ and  $\ext \! \Vol(\bar g)$ is the corresponding volume element.
Indeed, we shall  define higher Willmore energies  as follows (see~\cite{Will1}):
\begin{definition}\label{WillDef} Let the  dimension~$d$ of $M$ be odd. Then any functional of conformal embeddings given by
$$
E\big[\Sigma\hookrightarrow (M,\cc)\big]
=\int_\Sigma 
\ext \! \Vol(\bar g)\,
{\mathcal E}(g,\Sigma)\,   ,
$$
where~$g\in \cc$ and~$\bar g\in \cc_\Sigma$ is the corresponding induced metric,
is called a {\it higher Willmore energy} if
\begin{enumerate}[(i)]
\item the energy functional~${ E}$ only depends on the conformal embedding~$\Sigma\hookrightarrow (M,\cc)$, and 
\smallskip
\item \label{property} the functional gradient
of~$E$  with respect to variations of the embedding~$\Sigma\hookrightarrow M$
has leading linear term (in  any scale with non-vanishing mean curvature)
 proportional to~$\bar\Delta^{\frac{d-1}2} H^g$.
\end{enumerate}
\end{definition}

\noindent 
Note that if Condition~(\ref{property})
is suitably relaxed, it is possible to consider also generalizations of Willmore energies to even dimensional host spaces~$M$, see~\cite{Will1,hypersurface_old,JO}.

\smallskip

There are various 
 constructions  of higher Willmore energies in the literature. Indeed   Graham and Witten~\cite{GrWi} recovered the original Willmore energy in a study of submanifold observables in the AdS/CFT correspondence. 
 Higher Willmore equations were proposed in~\cite{WillB,Will1} and have subsequently been shown to arise as the Euler--Lagrange equation of higher Willmore energies~\cite{GrSing,RenVol}.
 We relate and discuss all of these energies  in Section~\ref{WE}. A key result, which may be established as a corollary of Theorem~\ref{Big_Formaggio} (a direct proof is given at the end of Section~\ref{Suitable}), is the following higher Willmore:

\begin{theorem}\label{AWm}
Let~$\Sigma^4\hookrightarrow (M^5,\cc_\Sigma)$ be a conformally embedded hypersurface.
Given~$g\in \cc$,  let 
${W\!\hh m}(g,\Sigma)$ be as given in Theorem~\ref{BQC}. Then 
 the functional  
$$
\int_\Sigma 
\ext \! \Vol(\bar g)\ {W\!\hh m}(g,\Sigma)\, ,
$$
is a higher Willmore energy.

\end{theorem}

\noindent Higher Willmore energies yield   anomalies in renormalized volumes of suitable conformally compact manifolds~\cite{RenVol} and submanifolds of Poincar\'e--Einstein structures~\cite{GrWi}. See Sections~\ref{Suitable} and~\ref{WE}.

 Our next application of Theorem~\ref{Big_Formaggio} is to compute the functional gradient of higher Willmore energies with respect to variations of the embedding $\Sigma\hookrightarrow M$. For that, we use a result of Graham~\cite{GrSing} (see also~\cite{RenVol}) that shows the
 functional gradient of the higher Willmore energy appearing on the right hand side of Equation~\nn{volanom} above is given by the obstruction $B|_\Sigma$ to smoothly solving the singular Yamabe problem given in Condition~(\ref{fruity}) of Theorem~\ref{Sc1}.
 Our result for this obstruction, and hence the variation of the energy~\nn{WE1}, is another consequence of Theorem~\ref{Big_Formaggio}  and is given in Theorem~\ref{Voldemort}.

\medskip

Our final application of Theorem~\ref{Big_Formaggio}
is to sixth order conformally invariant operators.
An old result of Graham is that there is no conformally invariant, sixth order, Laplacian power acting on conformal densities 
for general four dimensional conformal manifolds~\cite{GraNon}. 
 In a similar vein, for  two dimensional conformal structures, there is no conformally invariant fourth order, (Paneitz-type) Laplacian power operator. However there does exist an operator
$$
 \bar{\Delta}^2 +\cdots\, ,
$$
which maps $\Gamma(\ce \Sigma[1])\to \Gamma(\ce \Sigma[-3])$ that is defined for conformally embedded surfaces $\Sigma\hookrightarrow (M^3,\cc)$; see Lemma~\ref{Horcrux}.
This is the first operator in a sequence of extrinsically-coupled Laplacian powers~\cite{Will2} whose existence relies on
conformal embedding data. For the case of conformally embedded four-manifolds, there in fact exists a conformally invariant, sixth order, Laplacian power mapping $\Gamma(\ce \Sigma[1])\to \Gamma(\ce \Sigma[-5])$. This has leading term 
$$
 \bar{\Delta}^3 +\cdots\
\, ,
$$
and is defined along conformally embedded hypersurfaces $\Sigma^4\hookrightarrow (M^5,\cc)$, see
Theorem~\ref{Ronkonkoma}.
Graham's non-existence result is circumvented here because the additional embedding data 
provides the necessary tensors absent in the intrinsic conformal geometry of  $\Sigma^4$.

\medskip
As we post this work, we notice that Juhl has  announced results~\cite{J} for some of the scalar results presented in this manuscript.

\subsection{Riemannian conventions}\label{stRm}
Throughout we take $M$ to be a (for simplicity) oriented $d$-manifold. Unless specifically indicated, we take $d\geq 3$. Also, we will assume that
all structures are smooth. Complicated tensor expressions will often be displayed in an abstract index notation where, for example, $v^a$ and $\omega_a$ respectively denote sections of $TM$ and $T^*M$ (not the components of such in a choice of coordinate system). In particular, the contraction $\omega(v):=v^a \omega_a = v(\omega)$. Given a metric $g$, we use shorthand notations such as $v^2 := g(v,v)=g_{ab} v^a v^b=:v_a v^a$ and $u.v=g(u,v)=u^a g_{ab} v^b=: u_a v^b$, where the metric has been used to raise and lower indices according to the canonical isomorphism it determines between  $TM$ and $T^*M$. 
While we  work exclusively in Riemannian signature, many results carry over to indefinite signature metrics upon making the obvious adjustments.
The Riemannian curvature tensor $R$ for the Levi-Civita connection determined by $g$ is defined by
$$R(u,v)w = [\nabla_u, \nabla_v] w - \nabla_{[u,v]} w\, ,$$
where the bracket notation $[\pdot,\pdot]$ is used both for the commutator of operators and Lie bracket of vector fields.
In the abstract index notation, the vector $R(u,v)w$ is denoted
by $u^a v^b R_{ab}{}^c{}_dw^d$. We will also employ a hybrid notation $R_{uv}{}^c{}_w$ for contractions such as these. The Riemann tensor decomposes into its trace-free and pure trace pieces according to 
$$
R_{abcd}=W_{abcd}+2g_{c[a} P_{b]d}-2g_{d[a} P_{b]c}\, .
$$
In the above, $W_{abcd}$ is the {\it Weyl tensor} and $P_{ab}$ the {\it Schouten tensor}. The latter is related to the {\it Ricci tensor} $Ric_{ab}:=R_{ca}{}^c{}_b$ by
$$
Ric_{ab} = (d-2) P_{ab} + g_{ab}J\, ,\quad J:=P_{ab} g^{ab}\, .
$$
Also note that a bracketed group of indices denotes that they are totally skew so that, for example $u^{[a} v^{b]}:=\frac12( u^a v^b-u^b v^a)$. Similarly, $u^{(a} v^{b)}:=\frac12( u^a v^b+u^b v^a)$ and a $\circ$ will be used to indicate the trace-free part of a tensor, so that $u^{(a}v^{b)\circ}=u^{(a} v^{b)}-\frac 1d g^{ab} g(u,v)$.
The contraction of the metric~$g_{ab}$ and its inverse $g^{bc}$  gives the identity endomorphism of $\Gamma(TM)$ denoted by a Kronecker delta symbol $\delta_a^c=g_{ab} g^{bc}$.
Sometimes, more complicated contractions such as $W_{abcd}P^{bc}$ will be assigned shorthand notations such as $W_{aPd}$ or even $W(\pdot,P,\pdot)$ in situations where no ambiguity can arise.
To handle the endomorphism action of  the Riemann curvature two-form on tensors, the following notation is useful. For any rank two tensor $X_{ab}$ and an arbitrary tensor $Y^{cde\cdots}$, we define
$$
X^\hash Y^{cde\cdots}=
X^c{}_b Y^{bde\cdots}+
X^d{}_b Y^{cbe\cdots}+
X^e{}_b Y^{cdb\cdots}+\cdots\, .
$$
For higher rank tensors $X$ we use bullets to indicate which indices give the endomorphism, so for example we may write $(X_{\cdots a\bullet b\cdots c\bullet d \cdots})^\hash$ for this. 
The covariant curl of the Schouten tensor gives the {\it Cotton tensor}
$$
C_{abc}:=2\nabla_{[a} P_{b]c}\, .
$$
Finally, constants of proportionality never vanish, so the statement ``$A$ is proportional to $B$'' means that $A=kB$ for some constant $k\neq 0$.

\section{Background}\label{Bground}

We  study conformally embedded hypersurfaces~$\Sigma\hookrightarrow (M,\cc)$ where $(M,\cc)$ is an conformal manifold. (For simplicity we assume $M$ oriented.) To that end, the following result, is central:
\begin{theorem}\label{Sc1}
Let $(M,\cc)$ be a conformal $d$-manifold and $\Sigma\hookrightarrow M$ a smoothly embedded hypersurface. Then, given $g\in \cc$ there exists 
a function $s\in C^\infty M$ such that the following conditions hold:
\begin{enumerate}[(i)]
\item \label{cond1}$s|_\Sigma = 0 \neq \ext s\big|_\Sigma$.\\[-1mm]
\item Along $M\backslash \Sigma$, $Sc^{s^{-2} g}
=-d(d-1)\big(1+s^d {\mathcal B})
$,
where ${\mathcal B}\in C^\infty M$.\\[-1mm]
\item The function $s$ is unique modulo the addition of terms $s^{d+1}A$ for any $A\in C^\infty M$.\\[-1mm] 
\item\label{fruity} The function ${\mathcal B}|_\Sigma\in C^\infty \Sigma$ is uniquely determined by the conformal embedding data $\Sigma\hookrightarrow (M,\cc)$. Replacing $g$ by $\hat g=\Omega^2 g$ with $0<\Omega\in C^\infty M$, it obeys
$
{\mathcal B}(\hat g) = 
\Omega^{-d}
{\mathcal B}(g)
$.
\end{enumerate}
\end{theorem}
\noindent
In the above $\ext$ denotes the exterior derivative and $Sc^g$ the scalar curvature of the metric~$g$. 
The problem of finding a metric $g^o=s^{-2}g$ where $Sc^{g^o}=-d(d-1)$ and the function $s$ obeys Condition~(\ref{cond1}) above
is termed the {\it singular Yamabe problem}~\cite{Maz}.
The restriction of the function~$B$ to~$\Sigma$ is called the {\it obstruction density}. It is an example of a (local) {\it conformal hypersurface invariant}; generally these are defined as any section $V_\Sigma$ of a vector bundle over $\Sigma$ that is 
uniquely determined by the embedding data $\Sigma\hookrightarrow(M,\cc)$,  and subject to, for some {\it weight}~$w\in {\mathbb R}$,  the conformal covariance property
$$
V_\Sigma(\hat g) = \Omega^{w} V_\Sigma(g)\, ,
$$
where $\hat g=\Omega^2 g$ with $0<\Omega\in C^\infty M$; see~\cite{Will1} for the detailed 
definition. 
 
A version of the above theorem was  proved (in a slightly different setting) by Andersson, Chru\'sciel and Friedrich~\cite{ACF}. A second proof, given in~\cite{Will1}, shows how  the theorem can be employed to  study  conformal embeddings.  A first  key point  is that the uniqueness property of the function~$s$ implies that 
conformal hypersurface invariants can be encoded by its jets.
The result of~\cite{Will1}
relies on re-expressing the scalar curvature~$S\!\hh c^{\,s^{-2} g}$ (divided by $-d(d-1)$) as the ``square'' $h(I,I)$ of a section $I$ of a certain, rank $d+2$, vector bundle~$\ct M$ defined below.
This maneuver both gives a simple proof of the above theorem, and enables the use of the conformal boundary calculus of~\cite{GW}.

Before  defining the tractor bundle first recall, as alluded to in the introduction, that  
a conformal density~$\varphi$ of weight $w$ means an equivalence class $\varphi=[g;f]$ where
$(\hat g,\hat f) \sim (g,f)\in \varphi$ when $\hat g=\Omega^2 g$ and $\hat f=\Omega^w f$ for some $0<\Omega\in C^\infty M$. This is a section of an oriented associated line bundle $\ce M[w]\to M$, for the~${\mathbb R}_+$ group action on $t\in {\mathbb R}$ given by  $t\mapsto s^{-w} t$,  to the~${\mathbb R}_+$ principal bundle ${\mathcal Q}$ given by the ray subbundle of  $\odot^2 T^*M$ determined by the conformal class of metrics~$\cc$.
(At each $P\in M$, a pair of metrics $\hat g$ and $g\in \cc$ determine an element $s\in {\mathbb R}_+$ by the formula $s^d=\Vol(\hat g)(P)/\Vol(g)(P)$, where $\Vol(g)$ denotes the volume form of a metric $g$.)
Importantly, the {\it conformal metric} $\bm g_{ab}$ is the section of $\Gamma(\odot^2T^*M[2])$   tautologically defined by the equivalence class $\bm g_{ab}=[g;g_{ab}]$.  
Note that  the conformal density bundle $\ce M[w]$ is isomorphic to the $-(w/d)$th power of the bundle of volume forms ${(\mathcal Vol}M)^{-\frac wd}$  on $M$. 
Then the Levi-Civita connection 
extends to act on densities simply by writing these as sections $\varphi\in{(\mathcal Vol}M)^{-\frac wd}$.

Now, given a conformal $d$-manifold $(M,\cc)$ and a metric $g\in \cc$,  the standard tractor bundle~$\ct M$ 
is isomorphic to the rank $d+2$  
direct sum of vector bundles
$$
\ct M \stackrel{g}{\cong} \ce M[1]\oplus TM[-1]\oplus \ce M[-1]\, .
$$ 
Denoting a section $T$ of the above by $T^A\stackrel g=(\tau^+,\tau^a,\tau^-)$, changing to a conformally related metric $\hat g = \Omega^2 g$, the above isomorphism then changes according to 
$$T\stackrel {\hat g}=\big(\tau ^+,\tau^a+\Upsilon^a t^+,t^--\tau^a \Upsilon_a -\frac12 \Upsilon^a \Upsilon_a t^+\big)\, ,
$$
where $\Upsilon :=\ext \log \Omega \in \Gamma(T^*M)$
and indices are moved using the conformal metric.
Note that the first slot of $T$ in the above display defines a weight one
density $\tau^+\in \ce M[1]$
independently of the choice of $g$ giving the isomorphism.
When a weight one density $\tau>0$, we say that it is a {\it true scale}
because this determines the metric $g_\tau\in \cc$ for which $\tau = [g_\tau,1]$. 
When the weight one density determined by the first entry in $\Gamma(\ce M[1])$ of a standard tractor $T$ is a true scale, we say that $T$ {\it determines a scale}. 
The above display implies that the section $X\stackrel g= (0,0,1)\in \Gamma(\ct M[1])$ canonically  defines a tractor, 
termed the {\it canonical tractor}. Also, the above isomorphism between the tractor bundle and the direct sum given by a choice of $g\in \cc$ determines a pair of tractors~$Y\in \Gamma(\ct M[-1])$ and $Z\in \Gamma(T^*M\otimes \ct M[1])$---termed {\it injectors}---such that 
$$T^A\stackrel g=(\tau^+ ,\tau^a,\tau^-)= \tau^+ Y^A + Z^A_a \tau^a + X^A \tau^-\, .$$

The tractor bundle is naturally equipped with a  bundle metric $h:\Gamma(\odot^2 \ct M)\to C^\infty M$ defined acting on $T^A\stackrel g=(\tau^+,\tau^a,\tau^-)$ and $U^A\stackrel g=(\mu^+,\mu^a,\mu^-)$ by
$$
h(T,U)=h_{AB} T^A U^B \stackrel g=\tau^+ \mu^-+ {\bm g}(\tau,\mu) + \tau^- \mu^+\, .
$$
Note that $h$ is a fiberwise symmetric, non-degenerate but non-positive bilinear form; it is called the {\it tractor metric}. 
Weighted and tensor tractor bundles, denoted by $\ct^\Phi M[w]$, are defined in the obvous way. Here $\Phi$ labels some tensor product of the standard tractor bundle. The tractor metric is used to move abstract tractor indices in exactly the same way as for the metric and standard tensors.

The standard tractor bundle comes equipped with a canonical {\it tractor connection}~$\nabla^\ct$. Given $g\in \cc$, acting on a section $T^B\stackrel g=(\tau^+,\tau^b,\tau^-)$, this  is defined by
$$
\Gamma(T^*M\otimes \ct M)\ni \nabla^\ct_a T^B  \stackrel g= \big(\nabla_a\tau^+-\tau_a, \nabla_a \tau^b + \delta^b_a \tau^- + P^b_a\hh  \tau^+,\nabla_a \tau^- - P_a^b \tau_b\big)\, .
$$
On the right hand side above $\nabla$ is the Levi-Civita connection on densities described earlier and~$P$ is the Schouten tensor of $g$.

The tractor connection is particularly useful because of the following characterization of the Einstein condition (recall that a metric is {\it Einstein} iff multiplying it by some function yields the Ricci tensor; note that necessarily such a function is constant):
\begin{theorem}[Bailey--Eastwood--Gover~\cite{BEG}]
A conformal class of metrics admits an Einstein representative iff the standard tractor bundle equipped with its canonical tractor connection admits a parallel section that determines a scale.
\end{theorem}

There is also a very useful tractor characterization of the singular Yamabe problem. 
For that we need to introduce one further operator. To that end observe that the parallel condition~$\nabla^\ct T=0$ implies that
$$
T= \hat D \tau\, , 
$$
where the operator $\hat D:\ce M[1]\to \ct M$ is given by $1/d$ times  
%
the {\it Thomas  $D$-operator}, which is defined acting on weight $w$ tractors $T\in \ct^\Phi M[w]$ by
\begin{equation}\label{Dnolog}
D^A T\stackrel g=\Big(w(d+2w-2) T,
(d+2w-2) \nabla^\ct T, - \Delta^\ct T - wJ T\Big)
\, .
\end{equation}
Here the {\it tractor Laplacian} $\Delta^\ct:=\bm g^{ab} \nabla^\ct_a \nabla^\ct_b$. Also, for any weight $w\neq 1-\frac d2$, we denote~$\hat D:=(d+2w-2)^{-1} D$.
We term any standard tractor given by
$$
I_\sigma=\hat D \sigma\, ,
$$
where $\sigma\in \Gamma(\ce M[1])$, a {\it scale tractor} if
$I_\sigma$ is nowhere vanishing. In which case $\sigma$ is termed a {\it generalized scale}. 
Of special importance are generalized scales defined possibly in some neighborhood of a hypersurface $\Sigma$.
In particular
a generalized  scale $\sigma=[g;s]\in \Gamma(\ce M[1])$  with zero locus~$\Sigma$, 
such that any representative function $s$ is {\it defining} for $\Sigma$, meaning that $s$ obeys Condition~(\ref{cond1}) of Theorem~\ref{Sc1}, is termed a {\it defining density} for $\Sigma$. 
In these terms, the singular Yamabe problem is a normalization condition for scale tractors:
\begin{problem}\label{we have problems}
Given a conformal hypersurface embedding $\Sigma\hookrightarrow (M,\cc)$, find a defining density $\sigma$ for $\Sigma$ such that
$$
I^2_\sigma:=h(I,I)=1\, .
$$
\hfill$\blacksquare$
\end{problem}
\noindent
The equivalence of this problem to the singular Yamabe problem follows by noting that along~$M\backslash\Sigma$ one has
$
I^2_{[g;s]}\stackrel {s^{-2}g}=-d(d-1) Sc^{s^{-2}g}
$~\cite{Goal}. 
Theorem~\ref{Sc1} shows that Problem~\ref{we have problems} has an asymptotic solution 
\begin{equation}\label{stip}
I^2=1+\sigma^d B\, ,
\end{equation}
where $B$ is a weight $-d$ conformal density whose restriction to $\Sigma$ gives the obstruction density ${\mathcal B}_\Sigma$. In the case that the hypersurface $\Sigma$ is separating (so $M=M_-\sqcup \Sigma \sqcup M_+$), the existence of one-sided solutions along $M_+$ has been established by Loewner--Nirenberg when $M$ is a conformally flat sphere~\cite{Loewner} and for general structures in~\cite{Aviles,Maz,ACF}. 
Densities $\sigma$ solving Problem~\ref{we have problems} are termed {\it unit conformal defining densities} (we also use this terminology when only the data of $\sigma$ on $\Sigma\sqcup M_+$ is given). Densities $\sigma$ solving only Equation~\ref{stip} are termed \textit{asymptotic unit defining densities}.

Solutions to Problem~\ref{we have problems} only subject to $I^2=1+{\mathcal O}(\sigma^2)$ are also distinguished since then the tractor
$$
N:=I|_\Sigma \in \Gamma(\ct M|_\Sigma)
$$
gives an analog of the unit conormal---but now determined by the conformal embedding $\Sigma\hookrightarrow(M,\cc)$~\cite{Goal}. This recovers the  {\it normal tractor}~$N$ of~\cite{BEG}. 
For a choice of $g\in \cc$, one has $N\stackrel g= (0,\hat n,-H)$ where $\hat n$ and $H$ are the unit conormal and mean curvature determined by $g$.

The normal tractor can be used to establish an isomorphism between the bundle $\ct^\perp M|_\Sigma$ and~$\ct \Sigma$~\cite{Goal}, where $\perp$ refers to tractors $T$ along $\Sigma$ subject to $h(T,N)=0$ while $\ct \Sigma$ is the tractor bundle of the conformal structure $\cc_\Sigma$ induced along $\Sigma$ from $\cc$. For this reason the normal tractor is fundamental to the study of conformal hypersurface invariants.

We have now assembled all the ingredients required to build the extrinsic Paneitz operator in Display~\nn{magna_carta}. In fact, more generally, there exists a sequence of extrinsically-coupled Laplacian powers defined along conformally embedded hypersurfaces:

\begin{theorem}[Gover--Waldron~\cite{GW}]\label{Pk}
Let $\sigma$ be a unit conformal defining density for the embedding
$\Sigma\hookrightarrow (M,\Sigma)$,  and let $k\leq d-1$ be even. Then 
the operator
$$
P^{\Sigma\hookrightarrow M}_{k}:=
\Big[(I\csdot \hat D)^{k} \, \hh\scalebox{.5}{$\bullet$}\, \Big]\Big|_\Sigma 
$$
gives a well-defined map (termed an extrinsically-coupled conformal Laplacian power) 
$$
P^{\Sigma\hookrightarrow M}_{k}
:\Gamma\Big(\ct^\Phi M\Big[\frac{k-d+1}{2}\Big]\Big)\Big|_\Sigma\to
\Gamma\Big(\ct^\Phi M\Big[\frac{-k-d+1}{2}\Big]\Big)
\Big|_\Sigma\, ,
$$
determined by the local embedding data.
Moreover, $P^{\Sigma\hookrightarrow M}_{k}$ is given by a non-zero multiple of the~$k/2$ power of the hypersurface tractor Laplacian plus lower derivative order terms.
\end{theorem}

\begin{remark}
In~\cite{Will1},
the operator ${\sf P}_k^{\sss\Sigma\hookrightarrow M}$ was
 defined exactly as above, but the operator $(I\csdot D)^k$, 
 was used so that the
 above theorem extends to odd values of $k$, 
save for a modification of the final hypersurface Laplacian power statement.
In the same work, it was shown that for $k\geq d$, there exist 
extrinsically-coupled conformal Laplacian powers; these are discussed in Section~\ref{doIcogito}.
\end{remark}

When $k=4$ in the above theorem, the 
 extrinsically-coupled conformal squared Laplacian $\PanE$ becomes  the extrinsically-coupled Paneitz operator,  
 a formula for which, acting on general tractors along~$\Sigma$, is given  and proved in Section~\ref{prufs}. 

\smallskip

Theorem~\ref{Pk}
 embodies a general operator phenomenon: namely, when an operator ${\sf Op}$ acting on the section space of any weighted vector bundle ${\mathcal V}M[w]={\mathcal V}M\otimes \ce M[w]$ satisfies
 $$
 {\sf Op}\circ \sigma = \sigma \circ {\sf Op}'\, ,
 $$
for some 
${\sf Op}'$ is a smooth operator on $\Gamma({\mathcal V}M[w-1])$\, ,
and the unit conformal density is viewed as a multiplicative operator $\Gamma({\mathcal V}M[w-1])\to {\mathcal V}M[w]$, it holds that ${\sf Op}$ canonically gives a well-defined
operator on the space
$
\Gamma({\mathcal V}M[w])\big|_\Sigma\, .
$
In this case, we shall call the operator ${\sf Op}$ {\it tangential}. That the operator $ (I\csdot \hat D)^{k}$ acting on $\Gamma\big(\ct^\Phi M\big[\frac{k-d+1}{2}\big]\big)$ is  tangential
can be easily established using the~${\it sl}(2)$
solution generating algebra described in~\cite{GW}.

\begin{remark}\label{PIB}
Acting on functions $f\in \Gamma(\ce M_+[0])\cong C^\infty M_+$, the
Laplacian of the singular metric $g^o=\bm g/\sigma^2$
can be re-expressed in terms of the Thomas $D$-operator and the scale tractor $I$ of $\sigma$ according to
$$
\Delta^{g^o}\!f=-\sigma I\csdot D f\, .
$$
Notice that the operator $I\csdot D$ continues smoothly to the boundary, so it is propitious to study the boundary problem $I\csdot D f=0$ on  $M$ where  $f\in \Gamma(\ce M[0])$. More generally, when $f\in \Gamma(\ce M[w])$,  along $M_+$ we may consider $\varphi=f/\sigma^w \in C^\infty M_+$. Then
$$
\sigma^{w}
\Big(\Delta^{g^o}+\frac{w(d+w-1)}{d(d-1)}\hh Sc^{g^o}\Big)
 \varphi
 =
 -\sigma I\csdot D f 
 \, .
$$
%
%
%
%
%
%
In~\cite{GW}, the generalization 
$$I\csdot D\hh T=0$$
of the Laplace Equation~\nn{siren} to tractors $T\in \Gamma(\ct M[0])$ is studied. Also 
in that setting, the obstruction to smooth solutions with prescribed values of $T|_\Sigma$ was found to be given, for dimension~$d=5$  conformally compact structures,  by
$$
\Big(\frac1 {I^2} I\csdot D\Big)^4 T\big|_\Sigma\, .
$$ 
When the singular metric $g^o$ has constant scalar curvature $-d(d-1)$, the operator appearing above is proportional to the dimension-five, extrinscally-coupled Paneitz operator $\PanE$.
\end{remark}

Our first application of $\PanE$ to anomalies (in Section \ref{Suitable}) relies on an extension
of $\PanE$ to log densities; such densities are elaborated on
in~\cite{GW}.  The log density bundle of weight $w$ is the associated
line bundle ${\mathcal F}M[w]\to M$ to ${\mathcal Q}$ for the
${\mathbb R}_+$ group action on $t\in {\mathbb R}$ given by $t\mapsto
t- w \log s$. Sections are equivalence classes $[g;\ell]$ where $(\hat
g,\hat \ell) \sim (g,\ell)\in f$ when $\hat g=\Omega^2 g$ and $\hat
\ell=\ell + w \log \Omega$ for some $0<\Omega\in C^\infty M$. In
particular, for any weight $w=1$ conformal density $0<\tau=[g;t]\in
\Gamma(\ce M[1])$, $\log \tau = [g;\log t]$ is a weight one log
density. The Thomas-$D$ operator also maps weight $w$ log densities to
weight $-1$ tractors according to
\begin{equation}\label{Dlog}
D \lambda\stackrel{g}=
\Big((d-2)w,
(d-2) \nabla^{g} \lambda, - \Delta^{ g} \lambda - wJ^{g} \Big)\in \Gamma(\ct M[-1])
\, .
\end{equation}
In the above the Levi-Civita connection acting on a log density is defined analogously to the discussion above for conformal densities. Equivalently, a weight $w$ log density $\lambda$ can uniquely be expressed, in the obvious way, as $\lambda = \log \mu$ where $\mu$ is a strictly positive weight $w$ conformal density. Then $\nabla^{g}\lambda = 
(\nabla^{ g} \mu) /\mu\in\Gamma(\ce M[0]) $. We define $\hd \lambda$ by $D \lambda / (d-2)$.

\subsection{Conformal embeddings}
\label{conbed}

Here we summarize key aspects of the local invariant theory for
conformally embedded hypersurfaces $\Sigma\hookrightarrow (M,\cc)$; see~\cite{Will1,Will2,BGW} for a detailed discussion. Firstly, each~$g$ in $\cc$ determines a unit conormal~$\nu$
(uniquely by requiring that
this points towards $M_+$); this gives a one-form-valued weight one density $\hat n=[g;\nu] \in \Gamma(T^*M[1])|_\Sigma$.
Thanks to  the isomorphism between
$T\Sigma$ and the orthogonal complement 
of $\hat n$ in
  $TM|_\Sigma$, we may 
use the same abstract indices  for hypersurface tensors.
Then the weighted first fundamental form 
$$
\bar{\bm g}_{ab}={\bm g}_{ab}-\hat n_a \hat n_b \in \Gamma(T^*\Sigma[2])
$$
recovers the  conformal metric 
of $\Sigma$.
%
%
%

Given $g \in \cc$, the second fundamental form $\II_{ab}:=\bar g_a{}^c \nabla_c \nu^{\rm e}_b|_\Sigma$, where $\nu^{\rm e}$ is  any extension of the unit conormal (of $g$) to $M$. 
We  denote $\nabla^\top_a:=\bar g_a{}^c \nabla_c$ and, in general, use this notation when projecting with the first fundamental form. Given any extension $v_{\rm e}$ of $\bar v\in \Gamma(T\Sigma)$, the Gau\ss\ formula states 
\begin{equation}\label{GF}
\bar \nabla_a \bar v^b \eqSig {}\nabla^\top_a v_{\rm e}^b + \nu^b \II_{ac}  v_{\rm e}^c\, .
\end{equation}
Here and throughout bars are used to indicate hypersurface
quantities, such as the Levi--Civita connection $\bar \nabla$ of $\bar g$. The notation $\eqSig$ indicates equality along the hypersurface. 
The trace-free part of the second fundamental form yields a weight one, symmetric tensor-valued density 
$$
\IIo_{ab}\in \Gamma(\odot^2_\circ T^* \Sigma[1])\, .
$$

The equations of Gau\ss, Codazzi, Mainardi, and Ricci can be uniformly recast in terms of conformally invariant parts. When $d\geq 4$, the first three of these are 
\begin{align}
W^\top_{abcd} &= 
\bar W_{abcd} - 2 \IIo_{a[c} \IIo_{d]b} 
-
2\bar {\bm g}_{a[c} F_{d]b}
+
2\bar {\bm g}_{b[c} F_{d]a}\, ,
\nonumber
\\[1mm]
W^\top_{abc\hat n} &= 2\bar \nabla_{[a}\IIo_{b]c}-\tfrac2{d-2}\hh \bar{\bg}_{c[a} \bar \nabla\csdot \IIo_{b]}\, ,
\label{trfreecod}\\[1mm]
W_{\hat n ab \hat n}^\top &=\IIo^2_{(ab)\circ} -(d-3) F_{(ab)\circ}\, .\nonumber
\end{align}
The left hand sides of the the above display are   tensor densities of respective weights $2, 1$, and~$0$. Indeed, the trace-free covariant curl appearing on the right hand side of the second line above maps $\Gamma(\odot^2_\circ T^* \Sigma[1])$
to $\Gamma( T^*\Sigma \otimes_\circ \wedge^2 T^* \Sigma[1])$.
 For efficiency, we denoted here a density by the same symbol as an equivalence class representative for a given $g\in \cc$; the right hand side ought really read $\big[g;2\bar \nabla_{[a}\iio_{b]c}-\tfrac2{d-2}\hh \bar{g}_{c[a} \bar \nabla\csdot \iio_{b]}\big]$ for $\IIo=[g;\iio]$.
We will repeat this maneuver without comment in what follows.
The tensor density $F_{ab}\in \Gamma(\odot^2 T^*\Sigma[0])$ is the Fialkow tensor~\cite{Fialkow,Vyatkin}, given when $d\geq4$, for any $g\in \cc$ by $\Rho^\top-\bar \Rho + H\IIo +\frac12 \bar g H^2$.
Finally note that the second equation displayed above actually holds also when $d=3$. The identity is that the trace of the Fialkow tensor equals $\frac{1}{2(d-2)}\IIo^2=\frac{K}{2(d-2)}$; thus
\begin{align*}
\tfrac{1}{2(d-2)}\hh K&=J-P_{\hat n\hat n} -\bar J +\tfrac{d-1}2 H^2\, .
\end{align*}
Specialized to surfaces in ${\mathbb R}^3$, the above is Gau\ss'  {\it Theorema Egregium}.

\medskip

The trace-free parts of the second fundamental form 
and Fialkow tensor are examples of what we term conformal fundamental forms. Generally, these are rank two,  trace-free
densities determined by the conformal embedding $\Sigma\hookrightarrow M$ and encode normal derivatives of the ambient conformal metric along $\Sigma$. The question of which conformal fundamental forms can be determined canonically is subtle and interesting~\cite{BGW}.
When $d\geq 4$, canonical second and third fundamental forms always exist and are just given by $\IIo_{ab}\in \Gamma(\odot^2_\circ T^*\Sigma[1])$
and 
$$
\IIIo_{ab}:=-(d-3) \mathring F_{ab}\in \Gamma(\odot^2_\circ T^*\Sigma[0])\, ,
$$
respectively. 
For $d\geq 6$, 
the fourth fundamental form is
\begin{align}
\label{fourth}
\begin{split}
\IVo_{ab}
&:=-(d-4)(d-5) C_{\hat n(ab)}^\top  - (d-4)(d-5) H W_{\hat n ab \hat n} - (d-4) \bar \nabla^c W_{c(ab) \hat n}^\top 
\\[1mm]
&\phantom{=}
+ 2 W_{c \hat n \hat n (a}^{} \IIo_{b)\circ}^c + (d-6) \bar{W}^c{}_{ab}{}^d \IIo_{cd}
 - \tfrac{d^2 - 7d+18}{d-3} \IIIo_{(a} \csdot \IIo_{b)\circ}
 + \tfrac{d^3-10d^2+25d-10}{(d-1)(d-2)} K\IIo_{ab}
\\[1mm]
&
\in \Gamma(\odot^2_\circ T^*\Sigma[-1])\, .
\end{split}
\end{align}

\subsection{Hypersurface tractor calculus}

There are hypersurface tractor analogs of all the above extrinsic curvatures (at least for sufficiently large $d$). 
In particular for $d\geq 4$, in an obvious matrix notation, the {\it tractor second fundamental form}
\begin{equation}\label{LAB}
L^{AB}:\stackrel g=
\begin{pmatrix}
0&0&0\\
0&\IIo^{ab}&
-\tfrac1{d-2} \bar \nabla\csdot \IIo^a\\[1mm]
0& -\tfrac1{d-2} \bar \nabla\csdot \IIo^b &
\frac{\bar \nabla\cdot\bar\nabla \cdot \IIo + (d-2)\bar \Rho^{ab} \IIo_{ab}}
{(d-2)(d-3)}
\end{pmatrix}
\: \in \Gamma(\odot^2_\circ\ct \Sigma[-1])\, , 
\end{equation}
and for $d\geq 5$, the {\it Fialkow tractor}
\begin{equation}\label{FAB}
F^{AB}:\stackrel g=
\begin{pmatrix}
0&0&0\\
0&\Fo^{ab}&
-\tfrac1{d-3} \bar \nabla\csdot \Fo^a\\[1mm]
0& -\tfrac1{d-3} \bar \nabla\csdot \Fo^b &
\frac{\bar \nabla\cdot\bar\nabla \cdot \Fo + (d-3)\bar \Rho^{ab} \Fo_{ab}}
{(d-3)(d-4)}
\end{pmatrix}
\: \in \Gamma(\odot^2_\circ\ct \Sigma[-2])\,  .
\end{equation}
In dimensions $d=6$ and $d\geq 8$, the Fialkow tractor is related to the tractor second fundamental form by
\begin{align*}
(d-3) F_{AB} = \left( L^C_A L_{CB} - \tfrac{1}{d-1} K \overline{h}_{AB} - W_{NABN} \right) - \tfrac{1}{d-1} X_{(A} \hdb_{B)} K
- \tfrac{1}{(d-4)(d-5)} X_A X_B U \, .
\end{align*}
Note that the case $d=7$ is excluded from the above formula because the operator $\hdb$ acting on the rigidity density $K$
is not defined in that case.
A formula for the density $U\in \Gamma(\ce \Sigma[-4])$ in terms of extrinsic curvatures is given in~\cite{BGW}, but will not be needed here.
Once again, the bar notation has been used to indicate that $\bar h$ is the hypersurface's tractor metric and $\bar D$ its Thomas~$D$-operator.

The tractors displayed in Equations~\nn{LAB} and~\nn{FAB} are examples of a general insertion map from symmetric, trace-free, rank two tensor valued densities $t_{ab}$ to  tractors 
$$
\bar q : \Gamma(\odot^2_\circ \Sigma[w+2])\to
\Gamma(\odot^2_\circ \ct \Sigma[w])\, ,
$$
defined when $w\neq 1-d,2-d$
by
\begin{equation*}
\bar q(t)^{AB} 
\stackrel{g}= 
\begin{pmatrix}
0 & 0 & 0 \\
0 &t^{ab} & -\frac{\bar\nabla \cdot t^a}{d+w-1} \\
0 & -\frac{\bar \nabla \cdot t^b}{d+w-1} & \frac{\bar  \nabla \cdot \bar\nabla \cdot t + (d+w-1) \bar P_{ab} t^{ab}}{(d+w-1)(d+w-2)} \end{pmatrix}\, ,
\end{equation*}
which satisifies
$
\bar D_A T^{AB} =0= X_A  T^{AB} 
$.
When $d\geq 6$, we use this map to define the {\it tractor fourth fundamental form}
\begin{equation*}\label{JAB}
J^{AB}:=\bar q(\IVo)^{AB}\in \Gamma(\odot^2_\circ \ct \Sigma[-3])\, .
\end{equation*}

For $d\geq5$,  the tractor 
$W_{ABCD}\in \Gamma\big((\wedge^2 \odot_\circ  \wedge^2) \ct M[-2]\big)$,  given for $g\in\cc$ by
\begin{align*}
W_{ABCD}&\stackrel g=W_{abcd} Z^{a}_A 
Z^{b}_B 
Z^{c}_C 
Z^{d}_D 
+
4C_{abc} Z^{a}_{[A} 
Z^{b}_{B]} 
Z^{c}_{[C} 
X_{D]} 
+
4C_{abc} Z^{a}_{[C} 
Z^{b}_{D]} 
Z^{c}_{[A} 
X_{B]} 
\\&\phantom{=}+
\tfrac{4}{d-4}
B_{ac}Z_{[A}^a X_{B]}
Z_{[C}^c X_{D]}\, ,
\end{align*}
is termed the {\it $W$-tractor} (in~\cite{GOadv} this definition is made with an overall factor $d-4$) and has the algebraic symmetries of the Weyl tensors. 
Here the Bach tensor (for any $d\geq3$) is defined by
$$
B_{ab} :=  \Delta P_{ab} - \nabla^c \nabla_b P_{ca} + P^{cd} W_{adbc}\, .
$$
The 
Gau\ss-Thomas Equation of~\cite{BGW} relates the 
 bulk and hypersurface $W$-tractors for $d>5$ by 
 \begin{align}\label{GTE}
\begin{split} 
W^\top_{ABCD}
\big|_\Sigma =\, \overline{W}_{ABCD}& - 2 L_{A[C} L_{D]B} - 2 \overline{h}_{A[C} F_{D]B} + 2 \overline{h}_{B[C} F_{D]A} - \tfrac{2}{(d-1)(d-2)} \overline{h}_{A[C} \overline{h}_{D]B} K \\[1mm]
&+ 2 X_{[A} T_{B]CD} +2 X_{[C} T_{D]AB}  
- 2 X_A X_{[C} V_{D]B} + 2 X_B X_{[C} V_{D]A} \\
&+ \tfrac{1}{3(d-1)(d-2)} X_A X_{[C} \hdb_{D]} \hdb_B K - \tfrac{1}{3(d-1)(d-2)} X_B X_{[C} \hdb_{D]} \hdb_A K,
\end{split}
\end{align}
where
\begin{align*}
T_{ABC} &:= 2 \hdb_{[C} F_{B]A} + \tfrac{1}{(d-1)(d-2)} \overline{h}_{A[B} \hdb_{C]} K\:\in\: \Gamma(\ct \Sigma \otimes \wedge^2 \ct \Sigma[-3] )\, , 
\end{align*}
and $V_{AB}\in \Gamma(\odot^2 \ct \Sigma[-4])$ is a symmetric tractor built from curvatures such that $X^A V_{AB} = X_B V$ for some $V\in \Gamma(\ce M[-4])$ (see~\cite{BGW}). Here $\top$ denotes projection with respect to the {\it tractor first fundamental form}
$$
{\rm I}^A_B:=\delta^A_B -N^A N_B=\bar h^A_B\, ,
$$
where indices are raised with inverse tractor metric $h^{AB}$. 
Finally note that the relationship between the
hypersurface and bulk Thomas-$D$ operators is given in
Theorem~\ref{thesecondhorcrux}.

\section{Renormalized volume anomaly}\label{Suitable}

Let a closed 4-manifold~$\Sigma$
be conformally embedded as the boundary of a conformal 5-manifold~$(M^+,\cc)$ equipped with any unit conformal defining density $\sigma$. 
 The volume of~$M^+$ with respect to the singular metric $g^o={\bm g}/\sigma^2$ is ill-defined. However, given a choice of metric~$\bar g\in \cc_\Sigma$ 
 where~$\cc_\Sigma$
 is the 
 conformal class
of metrics induced by~$\cc$, 
a unique  renormalized volume~$\Vol_{\rm ren}(M^+,g^o,\bar g)$ can be defined~\cite{RenVol,GrSing}. 
Its dependence on the choice of~$\bar g\in \cc_\Sigma$ is encoded by a certain  ``anomaly operator'' which itself depends only on the local embedding data ~$\Sigma\hookrightarrow (M^+,\cc)$.
In particular, as proved in~\cite[Theorem 2.3]{VolumeII},  changing~$\bar g$ to~$\lambda^2 \bar g$ for some constant~$\lambda\in {\mathbb R}_{+}$ (and remembering to calibrate the definition of $Q_4^{\Sigma\hookrightarrow M}$ to that used previously), 
\begin{equation}\label{volanom}
\Vol_{\rm ren}(M^+,g^o,\lambda^2\bar g)-
\Vol_{\rm ren}(M^+,g^o,\bar g)
\, =\, 
 \frac{\log\lambda}{16} \int_{\Sigma} \ext\!\Vol(\bar g) \: Q_4^{\Sigma\hookrightarrow M}(g) \, .
\end{equation}
Remarkably the integrand on  the right hand side 
is the 
extrinsically-coupled~$Q$-curvature 
given in~\nn{creatine}.

\smallskip

The aim of this section is
to relate $Q_4^{\Sigma\hookrightarrow M}(g)$ to the already-computed quantity ${\mathcal Q}_4^{\Sigma\hookrightarrow M}(g)$.
 For that we need the following technical lemma.

\begin{lemma}\label{ironclad}
Let~$\tau$ be a true scale. Then 
the log density~$\log \tau$ obeys
$$
\log \tau = 
2 \lim_{\varepsilon \rightarrow 0} \frac{\tau^{\varepsilon/2} - 1}{\varepsilon}\quad
\mbox{ and } \quad
 \hd \log \tau = 
 2\lim_{\varepsilon \rightarrow 0} \frac{\hd \tau^{\varepsilon/2}}{\varepsilon}\, .
$$
\end{lemma}
\begin{proof}
The first identity  follows from 
the  easily-verified limit
$$
\log x = 
2 \lim_{\varepsilon \rightarrow 0} \frac{x^{\varepsilon/2} - 1}{\varepsilon}
$$ 
valid for any $0<x\in{\mathbb R}$.
The second identity relies on a direct computation in a choice of scale $g \in \cc$: from Equation~\nn{Dlog} we have 
$
\hat D \log\tau \stackrel g = 
\big(1,\tau^{-1}\nabla \tau,(d-2)^{-1} (|\nabla \tau|^2 -\tau\Delta \tau  - \tau^2 J)/\tau^2\big) 
$, while Equation~\nn{Dnolog} 
gives 
$$(2/\varepsilon)\hh \tau^{-\varepsilon/2}\hat D \hh\tau^{\varepsilon/2} \stackrel g = 
\Big(1,\tau^{-1}\nabla \tau,(d+\varepsilon-2)^{-1} \big[(1-\tfrac\varepsilon2)\hh |\nabla \tau |^2 -\tau\Delta \tau  - \tau^2 J\big]/\tau^2\Big)\, . $$
Taking the limit $\varepsilon\to0$, the result then follows.
\end{proof}

\noindent
A corollary of the above is the desired result:

\begin{lemma}\label{fiatlux}
Let~$\sigma$ be a unit conformal defining density for~$\Sigma \hookrightarrow (M^5, \cc)$ and let~$\tau$ be a true-scale. Then,
$$ Q^{\Sigma \hookrightarrow M}_4 (g_\tau) 
=
{\mathcal Q}^{\Sigma \hookrightarrow M}_4 (g_\tau) 
\,.$$
\end{lemma}

\begin{proof}
From Lemma~\ref{ironclad} we have that
$$
I\csdot \hat D \log \tau = 2\lim_{\varepsilon\to 0} \frac{I\csdot \hat D \tau^{\varepsilon/2}}\varepsilon\, .
$$
A similar argument to that employed in the proof of that lemma also establishes that
$$
I\csdot \hat D^4 \log \tfrac1\tau = 2\lim_{\varepsilon\to 0} \frac{I\csdot \hat D^4 \tau^{-\varepsilon/2}}\varepsilon\, .
$$
Thus  it follows that 
\begin{align} \label{Q4a}
Q_4(g_\tau) = \lim_{\varepsilon \rightarrow 0} \frac{2}{\varepsilon} I\csdot\hd^4 \tau^{-\varepsilon/2} \big|_{\Sigma,g_\tau}\, .
\end{align}
(Recall, that a true scale $\tau$ determines a Riemannian metric $g_\tau={\bm g}/\tau^2$.)

To extract~${\mathcal Q}^{\Sigma \hookrightarrow M}_{4}(g_\tau)$ from Corollary~\nn{forIamamerecorollary}, note that the basis of invariants present in the expression for~$P^{\Sigma \hookrightarrow M}_4 \tau^{\frac{5-d}{2}}$ are stable as the dimension~$d$ varies~\cite{GPt} and their coefficients are rational functions of~$d$, so we can treat the formula given in the corollary as a universal formula for~$P^{\Sigma \hookrightarrow M}_4 \tau^{\frac{5-d}{2}}$ with~$d$ an arbitrary parameter. This ``dimensional continuation''-type argument is standard, see for example~\cite{GOpet}. Then, working in the~$\tau$-scale, we can write
$${\mathcal Q}^{\Sigma \hookrightarrow M}_4(g_\tau) \stackrel{g_\tau}= \lim_{d \rightarrow 5} \left(\frac{2}{d-5} P^{\Sigma \hookrightarrow M}_4 \tau^{\frac{5-d}{2}} \right)\,.$$
Substituting $5+\varepsilon$ for $d$ in the universal formula, we obtain Equation~\nn{Q4a}, as required by the lemma.
\end{proof}

\medskip

Assuming Corollary~\ref{forIamamerecorollary} (whose proof is given in Section~\ref{prufs}), we can now prove the extrinsically-coupled $Q$-curvature Theorem~\ref{BQC}.

\begin{proof}[Proof of Theorem~\ref{BQC}]
Following Lemma~\ref{fiatlux}, we may set~$d =5$ in the expression for ${\mathcal Q}_{d-1}^{\Sigma\hookrightarrow M}(g)$ in Equation~\nn{ext-Q} 
to obtain the quoted result for ${Q}_{d-1}^{\Sigma\hookrightarrow M}(g)$.

It only remains to establish that the quantities~$W\! m$ and $U$ are indeed conformal invariants as claimed. To do so, we first consider the quantity the weight~$-4$ hypersurface density~$L_{AB}\bar \square_Y L^{AB}$ where the  hypersurface Yamabe operator is defined acting on weight~$-1$ four-manifold tractors~$T$ via
$$
- \bar D^A T = X^A\bar  \square_Y T\, ,
$$
and $L_{AB}$ is the tractor second fundamental form. It is not difficult to compute (see~\cite{Abrar}) that
~$$
\bar \square_Y L^{AB}\stackrel {\bar g}=
\begin{pmatrix}
0& 0 & \star\\[1mm]
0 & 
 \bar{\Delta} \IIo^{ab}
-\tfrac43  \hh\bar{\nabla}^{(a} \bar{\nabla} \csdot \IIo^{b)\circ} 
 -4\hh \bar{P}^{(a} \csdot \IIo^{b)\circ}
 - \hh\bar{J} \IIo^{ab}
 +\star\,\,  \bar g^{ab}
&\star\\[1mm]
\star&
\star&
\star
\end{pmatrix}\, .
~$$
 Here~$\star$ denotes terms that do not contribute to the (manifestly invariant) density~$
L_{AB} \bar \square_Y L^{AB}
$. Applying the Codazzi--Mainardi Equation~\nn{trfreecod}, we find 
\begin{equation}\label{LboxL}
L_{AB} \bar \square_Y L^{AB} = \IIo^{ab} \bar{\nabla}^c W_{cab \hat{n}}^\top + \IIo^{ad} \IIo^{bc} \bar{W}_{abcd}\,.
\end{equation}
Thus, because $\IIo^{ad} \IIo^{bc} \bar{W}_{abcd}$ is conformally invariant, so too is $\IIo^{ab} \bn^c W_{cab \hat n}^\top$. Thus we have that~$U$ is composed solely of manifestly invariant scalars.

To handle $W\!\hh m$, we first define
$$W_0 := W\!\hh m + 24 \Fo^2 + 6 \IIo \csdot \Fo \IIo + \tfrac{31}{12} K^2 + \tfrac{11}{2} \IIo^{ad} \IIo^{bc} \bar{W}_{abcd} - \tfrac{13}{2} \IIo^{ab} \bn^c W_{cab \hat n}^\top\,.$$
Comparing $W_0$ with $\Kdd := I \csdot \hd^2 (\hd I)^2$ from~\cite{BGW} reveals that $W_0 = \tfrac{3}{2} \Kdd$, which is manifestly invariant and thus $Wm$ is also invariant.
\end{proof}

\begin{remark}
Equation~\nn{LboxL} 
in the proof of Theorem~\ref{BQC} 
also establishes that~$\bar{\nabla}^c W_{c(ab) \hat{n}}^\top$ is invariant
in~$d=5$ dimensions; 
alternatively this is the application to $W^\top_{c(ab)\hat n}$ of a well-known invariant first order operator on conformal 4-manifolds.
\end{remark}

Finally, Theorem~\ref{BQC} together with the fact that 
 $$ 
 -\frac1{24} \int_\Sigma \ext\! \Vol(\bar g) \IIo_{ab} \bar \Delta \IIo^{ab}
  =\frac12
 \int_\Sigma \ext\! \Vol(\bar g)
\, \big| \bar \nabla H|^2_{\bar g}+
\mbox{subleading hypersurface derivative terms}\, ,
$$
establishes that
the integral of the density $W\!m$ is a higher Willmore energy
 for hypersurface embeddings in generally curved  
 conformal five-manifolds; this proves Theorem~\ref{AWm}.

\section{Willmore energies}\label{WE}

In principle, higher Willmore energies can be constructed by writing
 down a list of possible extrinsically-coupled Riemannian invariants and fitting coefficients in order to satisfy Definition~\ref{WillDef}. Such a computation was performed by Guven in a study of 
bending energies for  hypersurfaces embedded in~${\mathbb R}^5$ (with its standard Euclidean metric $\delta$) and invariant under  M\"obius transformations~\cite{Guven}.
Correcting for a typographical error (see~\cite{GR})  Guven's energy is given in our notations by
$$
E_{\rm Gu}[\Sigma\hookrightarrow ({\mathbb R}^5,\cc_{\rm std})]\stackrel \delta=\frac12
\int_\Sigma \ext\! \Vol(\bar g)
\Big(\big| \bar \nabla H|^2_{\bar g}
-H^2 K + 3 H^4
 \Big)
  \, .
$$

\medskip

 For 
 closed four-manifolds, $-\int_\Sigma W\!\hh m$
 gives
 $$
E_{\rm Wm}[\Sigma\hookrightarrow ({\mathbb R}^5,\cc_{\rm std})]:=\frac12
\int_\Sigma \ext\! \Vol(\bar g)
\Big(\big| \bar \nabla \IIo|^2_{\bar g}
 +12 \IIo \csdot\hh C_n^{\top}
  - 8 \IIo \csdot \bar{P} \csdot \IIo
  + 7 \bar{J} K 
 +12 H \IIo^3
  - 24 H \IIo \csdot \Fo
  \Big)\, .
$$
Each term in the above integrand is required for  this energy to be conformally invariant, while, as proved in~\cite[Theorem 5.1]{Will1},  the variation of first term has leading linear term proportional to~$\bar\Delta^2 H$, and
by inspection, none of the other terms have this property. This establishes Theorem~\ref{AWm}.

%
\smallskip

There currently exist two independent constructions of higher Willmore energies based on anomalies. 
For the first of these, consider 
as motivation a surface~$\Sigma$  conformally embedded in 
the boundary of a Poincar\'e--Einstein manifold~$(X,g^o)$ whose boundary~$\partial X=M$ has conformal structure~$\cc$.  Graham and Witten~\cite{GrWi}  observed that the anomaly in the renormalized volume of a three manifold~$\Xi$, with boundary~$\Sigma$,
 embedded minimally in the bulk 
 was given by the Willmore energy of the surface embedding~$\Sigma\hookrightarrow  (M,\cc)$. 
Upon replacing the surface~$\Sigma$ with
a conformally embedded 4-manifold and~$\Xi$ with a minimally embedded 5-manifold
in the above discussion, the anomaly is a higher Willmore energy that
has been computed by Graham and Reichert~\cite{GR}. Note also that four-dimensional Willmore-type energies have recently appeared in
 more general physical contexts, including higher codimensional settings~\cite{Chalabi}. For the case that~$M$ is a 5-manifold, expressed in our notations, the Graham and Riechert formula for a higher Willmore energy is
\begin{multline*}
E_{\rm GR}[\Sigma\hookrightarrow (M,\cc)]
= \frac{1}{16} \int_\Sigma \ext\!\Vol(\bar g) \:  \big( Q^{\Sigma}_4 
- \tfrac{1}{3} W\! m \\[1mm]
+ \tfrac{1}{6} \IIo^{ab} \bar{\nabla}^c W_{cab \hat{n}}^\top +\tfrac16 \IIo^{ad} \IIo^{bc} \bar{W}_{abcd}
+ \tfrac{1}{24} K^2 
- 2 \IIo^2 \csdot \Fo 
+ 2 \Fo^2
\big) \, .
\end{multline*}
In fact, in~\cite{AGW} it is shown that the above energy functional is the integral of an extrinsically-coupled~$Q$-curvature. Let us give some details: If the unit conformal defining density~$\sigma$ determines the Poincar\'e--Einstein metric~$g^o$
and~$\mu$ is a unit conformal defining density for the conformal embedding of~$\Xi$ in~$(X,\cc_X)$, then (see~\cite{AGW}) the minimal hypersurface condition reads
$$
I_\sigma\csdot I_\mu \stackrel\Xi = 0\, .
$$
Unfortunately
 however,  the metric~$g^o$ pulled back to~$\Xi$ does not in general have constant scalar curvature, so Equation~\nn{creatine} cannot be directly applied. Nonetheless an analog of the Laplace--Robin operator~$L_\sigma^T$ along~$\Xi$ 
suitable for the construction of~$Q$-curvature quantities does exist. Let us  generalize this picture to the case~$\dim X=d$ (so $\dim M=d-1$ and $\dim \Sigma=d-2$. Then acting on a weight~$w\neq 1-\frac d2, 2-\frac d2$ density~$f_\Xi\in \Gamma(\ce \Xi[w])$,   this operator is given by
$$
L_\sigma^T f_\Xi = \Big(
\frac{d+2w-3}{d+2w-2}\hh  I_\sigma \csdot D - \frac wd I_\mu \csdot D (I_\sigma\csdot I_\mu) + 
\frac\sigma {(d+2w-4)(d+2w-2)} I_\mu\csdot D^2
\Big)f\Big|_\Xi\, ,
$$
where~$f$ is any extension of~$f_\Xi$ to~$X$. In these terms, the anomaly in the renormalized area of $\Xi$ gives a higher Willmore 
$$
E_{\rm AGW}[\Sigma\hookrightarrow M \hookrightarrow (X,\cc_X)]
=\int_\Sigma 
\ext \! \Vol(\bar g)\hh
Q_{d-2}^{\Sigma\hookrightarrow M\hookrightarrow X}(\bar g)\, ,
$$
where this extrinsically-coupled~$Q$-curvature~$Q_{d-2}^{\Sigma\hookrightarrow M\hookrightarrow X}(\bar g)$ is given by
$$
Q_{d-2}^{\Sigma\hookrightarrow M\hookrightarrow X}\!(\bar g)
=\big(L_\sigma^T \big)^{d-2} \log\tfrac1\tau \big|_\Sigma\, .
$$
Here~$\tau$ is any true scale on~$X$ and~$\bar g$ is the metric on~$\Sigma$ determined by this.
Moreover, the above formula still generates the 
anomaly
 in the more general case where~$g^o$ is a singular Yamabe metric, rather than a Poincar\'e--Einstein one. In the latter case the  above quantity integrated over~$\Sigma$ only depends on the embedding~$\Sigma\hookrightarrow (M,\cc)$ (so is insensitive to the embedding of~$M$ in~$(X,\cc_X)$) and specializing to $d=6$, up to an overall coefficient, gives a holographic formula for   the Graham--Reichert result.

The second anomaly-based Willmore construction is the one given in the previous section based on  the data of a four manifold $\Sigma$ embedded in a five-dimensional conformal manifold~$(M,\cc)$. 
Indeed the renormalized volume anomaly in Equation~\nn{volanom} is a higher Willmore energy
$$
E_{\rm GW}[\Sigma\hookrightarrow (M,\cc)]=
\int_\Sigma \ext \! \Vol(\bar g)
\big( Q^{\Sigma}_4 
+W\! m 
+U\big)
\, .
$$
This result was proved in~\cite{BGW}.
The above display is the  integral of the extrinsically-coupled $Q$-curvature computed in Theorem~\ref{BQC} for closed $\Sigma$.

\medskip

The Gau\ss--Bonnet theorem applied to the four dimensional hypersurface $\Sigma$ gives
$$
\int_\Sigma  \ext\! \Vol(\bar g)
 Q_4^\Sigma
 =
 8\pi^2 \chi(\Sigma)
-\frac14 \int_\Sigma  \ext\! \Vol(\bar g)
|\bar W|^2
\, .
$$
The first term on the right hand side is topological and thus embedding independent, while the embedding variation of the second term factors
through the Bach tensor of $(\Sigma,\cc_\Sigma)$. Hence
 any amount~$\int Q_4^\Sigma$ can be added to an energy while maintaining its Willmore property. Both  anomaly constructions manifest this property. Also, just as only the integral of the conformally invariant squared Weyl tensor is needed to write an action whose metric variation yields the Bach tensor, 
it is also possible to construct higher Willmore energies whose integrand itself is conformally invariant. Several such constructions exist, an example of which is $-\int_\Sigma \ext\! \Vol(\bar g) \hh W\!m$ as given above, another of which is below.


\medskip

There are two interesting Willmore constructions that rely on tractor hypersurface calculus. The first, due to Vyatkin~\cite{Vyatkin}, relies on a certain  $Q$-operator acting on closed one-forms that can in principle be twisted by any connection~\cite{BGdeRham,GodeRham}. In particular, by twisting with the tractor connection of~$(\Sigma,\cc_\Sigma)$ and acting on the projecting part $L_a^A\in\Gamma(T^*\Sigma\otimes \ct \Sigma)$
of the tractor second fundamental form, and then integrating the result against $L^a_A$,  produces
a higher Willmore energy for a four manifold~$\Sigma$ embedded in $({\mathbb R}^5,\cc_{\rm std})$:
$$
E_{\rm Vy}[\Sigma\hookrightarrow ({\mathbb R}^5,\cc_{\rm std})]\stackrel \delta=\frac12
\int_\Sigma \ext\! \Vol(\bar g)
\Big(\big| \bar \nabla H|^2_{\bar g}
-H^2 K + 3 H^4
 +\frac5{24} K^2 - \frac14 \IIo^4
 \Big) -2\pi^2 \chi(\Sigma)
  \, .
$$ 
Our final higher Willmore construction 
is based on the action of the extrinsically-coupled Paneitz operator on the normal tractor~\cite{Will1}. For that we need the following result. 
\begin{theorem}
\label{P4N}
Let~$d = 5$. Then, given $g\in \cc$, 
\begin{align*}
\Gamma\big(\ct M[-4]\big|_\Sigma\big)\ni
\PanE N^A \stackrel g=  
N^A\hh  {A}
+ X^A \, {B}
 + \bar{Z}^A_a \,  {C}^a 
 +\bar{Y}^A \hh{D}
 \, ,
\end{align*}
where
\begin{align*}
A \
=&- \tfrac23 \hh {W\!\hh m} 
-\tfrac{23}{6} K^2
- \tfrac{11}{3} \IIo^{ad} \IIo^{bc} \bar{W}_{abcd}
+ \tfrac{13}{3} \IIo^{ab} \bn^c W_{cab \hat n}^\top
- 4 \IIo^2 \csdot \Fo
- 16 \Fo^2\, ,\\
{C}_a =&\:\,   8 \bar{\nabla}_a (\IIo \csdot \Fo) + \tfrac{10}{3} \IIo_a \csdot \bar{\nabla} K + \tfrac{20}{9} K \bar{\nabla} \csdot \IIo_a\,,\\
D \ =& - 24 \IIo \csdot \Fo  \,,
\end{align*}
and the leading hypersurface derivative term 
of the scalar $B$ is $-\frac13 \bar \Delta \bn\csdot\bn\csdot\IIo$. 
\end{theorem}
\noindent
This theorem is a corollary of Theorem~\ref{Big_Formaggio} and is proved in Section~\ref{prufs}. The result for the scalar~$B$ is implicitly given there and will be used to compute the 
obstruction of Theorem~\ref{Sc1}.
 As a direct consequence of this theorem, we obtain the following corollary (see also~\cite{Will2}):
\begin{corollary}
Let $d = 5$. Then, $$
E_{\rm NN}[\Sigma\hookrightarrow (M,\cc)]:=
\int_{\Sigma}\ext\! \Vol(\bar g)\hh  N_A \PanE N^A$$ is a higher Willmore energy and, given~$g \in \cc$,
$$N_A \PanE N^A  \stackrel g= 
\hh -\tfrac23 \hh {W\!\hh m} 
-\tfrac{23}{6} K^2
- \tfrac{11}{3} \IIo^{ad} \IIo^{bc} \bar{W}_{abcd}
+ \tfrac{13}{3} \IIo^{ab} \bn^c W_{cab \hat n}^\top
- 4 \IIo^2 \csdot \Fo
- 16 \Fo^2\,.$$

\end{corollary}

\begin{remark}
In a conformally flat background, the various higher Willmore energies
discussed above can all be expressed as linear combinations of the integrals of $W\!\hh m$, $K^2$, $\IIo^4$ as well as the Euler characteristic. It is natural to normalize any energy such that it takes the form $\frac12 \int (\bar\nabla \IIo)^2+\cdots$, 
and therefore write any given energy $E_\bullet$
as 
$$
E_\bullet = \int_\Sigma 
\ext\! \Vol(\bar g)\hh
\big(\!- W\!m + \alpha  K^2 + \beta  \IIo^4\big) + 8\pi^2 \gamma\hh \chi(\Sigma)\, .
$$
The coefficients $\alpha$, $\beta$ and $\gamma$ for the various generalized energies
are summarized in the table below (whose first line is the same as~\cite[Lemma 10.3]{Jagain}):
$$
\begin{array}{l||r|r|r}
&\alpha \ &\beta \ &\gamma\\
\hline\hline
&&&
\\[-3mm]
E_{\rm Gu}& -\tfrac{11}{6} \ &\tfrac{5}{2} \ &3\\[2mm]
E_{\rm Wm}&0\ &0\ &0\\[2mm]
E_{\rm GR} & -\tfrac{11}{6} \ &\tfrac{5}{2} \ & 3\\[2mm]
E_{\rm GW}& \tfrac{9}{2}\  &- \tfrac{31}{2} \ & -1\\[2mm]
E_{\rm Vy}  &-\tfrac{7}{12} \ &1\ &0\\[2mm]
E_{\rm NN} & \tfrac{35}{12} \ &-20 \ &0
\end{array}
$$
\smallskip

\noindent
Note, that in the above table we used that for conformally flat dimension five ambient spaces 
$$
\int_\Sigma \ext\! \Vol(\bar g)
W\!m
=\int_\Sigma \ext\! \Vol(\bar g)
\Big(-6 \big| \bar \nabla H|^2_{\bar g}
+6H^2 K -18 H^4
-\frac{11}{6} K^2 + \frac52 \IIo^4
 \Big) +24\pi^2 \chi(\Sigma)\, .
$$
To facilitate comparison with~\cite{Vyatkin}, we record further identities valid for conformally flat dimension five ambient spaces
$$
\int_\Sigma \ext\! \Vol(\bar g)\IIo^{ab}\bar \nabla_a \bar \nabla\csdot \IIo_b = 
-9\int_\Sigma \ext\! \Vol(\bar g)\big| \bar \nabla H|^2_{\bar g}
=-\int_\Sigma \ext\! \Vol(\bar g)|\bar\nabla  \csdot \IIo|^2_{\bar g}\, ,
$$
and $\bar P_{ab} = H\IIo-\frac12 \bar g_{ab} H^2 
+\frac1{12} \bar g_{ab} K -\frac12 \IIo^2_{ab}$.
\end{remark}

\section{Willmore functional gradients}

\label{NinaSimone}

An interesting problem is to compute the functional gradient
of higher Willmore energies. 
Such computations are somewhat intricate, but for the Willmore energy based on the anomaly of the renormalized volume for a singular Yamabe metric,  
$$
E_{\rm Sing}:=
\int_\Sigma \ext\! \Vol(\bar g)\hh  Q^{\Sigma\hookrightarrow M}_{d-1}(\bar g) \, ,
$$
there is a streamlined holographic method. In~\cite{WillB} it is observed that the obstruction to smoothly solving the singular Yamabe problem computed for three manifolds  in~\cite{ACF} gives the Willmore invariant (or equivalently the functional gradient)  of the Willmore energy for
conformal surface embeddings.
Subsequently Graham proved that this relation between singular Yamabe obstructions and functional gradients is a general one that holds in any dimension~\cite{GrSing}; see also the alternate approach~\cite{RenVol} based on a distributional calculus of~\cite{hypersurface_old}. The key result is that the functional gradient of~$E_{\rm Sing}$ is given by the obstruction density~${\mathcal B}_\Sigma\in \Gamma(\ce \Sigma[-5])$ corresponding to the function~${\mathcal B}|_\Sigma$ in Condition~\nn{fruity}.
For an explicitly given asymptotic singular Yamabe metric, thanks to Equation~\nn{stip}, the obstruction density can easily be computed by calculating the square of the corresponding scale tractor.
To write a general formula, for any conformal embedding, a result of~\cite{Will2} shows that~${\mathcal B}_\Sigma$ can be computed from knowledge of how the extrinsically-coupled Paneitz operator acts on the normal tractor. In~$d=5$ dimensions,
\begin{equation}
\label{FORMAGGINO}
{\mathcal B}_\Sigma = \tfrac{1}{1440}  (\bar{D}_A \circ \top) \big(\PanE N^A - 3 I \csdot \hd^3 (X^A K_{\rm e}) \big) \, .
\end{equation}
Here $K_{\rm e}$ is the canonical extension of the 
rigidity density given by
$$
K_{\rm e}:=(\hd I)^2\in \Gamma(\ce M[-2])\, .
$$
We also employ the notation $\dot f := I\csdot \hd f$ for any density $f\in \Gamma(\ce M[w])$ where $w\neq 1-\frac d2$.
We now need the following result.
\begin{lemma}\label{dots}
Let $d=5$ and $K_{\rm e}$ be defined as above. Then along $\Sigma$,  the weight $-5$ density $\dddot K_{\rm e}$ is given by the following sum of conformally  invariant terms
\begin{align*}
\dddot K_{\rm e}\big|_\Sigma 
=& -2 L^{ab} \big(16 \IIo^3_{(ab)\circ} 
+24  \IIo_{(a} \csdot \Fo_{b)\circ}
-5  K \IIo_{ab} \big)
- 48 W_{ABCN}^\top \hdb^A F^{BC}\\&
- 42 K \IIo \csdot \Fo
- 48 \IIo \csdot \Fo^2 
-24 \IIo^{bc} (5 \Fo^{ad} + W_{\hat{n}}{}^{ad}{}_{\hat n}) \bar{W}_{abcd}
- 48 \IIo^a_b W_{acd \hat{n}}^\top W^{bcd}{}_{\hat n}^\top \\&
+ 8 \IIo \csdot \bar{B}
-24 (\Fo^{ab} + W_{\hat n}{}^{ab}{}_{\hat n}) \bn^c W_{cab \hat{n}}^\top\, , 
\end{align*}
where $L^{ab}$ defines the mapping
$$
\Gamma(\odot^2_\circ T^*\Sigma[-1])
\ni X_{ab}\longmapsto 
\bn\csdot\bn\csdot X
+\bar P\csdot X
\in 
\Gamma(\ce \Sigma[-5])\, ,
$$
and
$$
W_{ABCN}^\top \hdb^A F^{BC}=
 -\Fo\csdot C^\top_{\hat n}
+W_{abc\hat n}^\top  \bn^a \Fo^{bc}
-H \IIo^2\csdot \Fo
+2 H \Fo^2\, .
$$
\end{lemma}

\noindent
The proof of this result is given in Section~\ref{prufs}.
Applying Equation~\nn{FORMAGGINO}, Lemma~\ref{dots}, and Theorem~\ref{P4N} we obtain the following result. 
\begin{theorem}\label{Voldemort}
Let $d=5$. Then the obstruction density is given by the sum of conformally invariant terms \begin{align*}
{\mathcal B}_{\Sigma} = -\tfrac{1}{120} \big[
Ob &
+16  L^{ab}\big(  \IIo^3_{(ab)\circ} 
+ \tfrac{9}{4}\IIo_{(a} \csdot \Fo_{b)\circ}
+ \tfrac{1}{96} K \IIo_{ab} \big) 
+ 54 W_{ABCN}^\top \hdb^A F^{BC}\\&
+ 3 (4 W_{\hat n}{}^{ab}{}_{\hat n} + 3 \Fo^{ab}) \bn^c W_{cab \hat n}^\top 
- 13 \IIo \csdot \bar{B}
+ (12 W_{\hat n}{}^{ad}{}_{\hat n} + 69 \Fo^{ad}) \IIo^{bc} \bar{W}_{abcd} \\&
+ 24 \IIo^a_b W_{acd \hat{n}}^\top W^{bcd}{}_{\hat n}^\top
+ 60 \IIo \csdot \Fo^2
+ \tfrac{89}{2} K \IIo \csdot \Fo
 \big] \, ,
\end{align*}
where
\begin{align*}
Ob
:=&\phantom{+}{{ \bar{\Delta} \bar{\nabla} \csdot \bar{\nabla} \csdot \IIo}}
+6\,{ \bar{\nabla} \csdot B_{\hat n}^\top}
+6\,{ \IIo} \csdot { B^\top}
-6\,{ \bar{P}} \csdot { C_{\hat n}^\top} \\
&-6\,{ \bar{P}}{}^{ab} { \bar{\nabla}^c W_{{c a b \hat{n}}}^\top } 
+9\,\bar{P} \csdot \bar{\Delta} \IIo
-{{{ \bar{J}}\,{ \bar{\nabla} \csdot \bar{\nabla} \csdot  \IIo}}}
+6\,{ (\bar{\nabla} \IIo)} \csdot {( \bar{\nabla} \bar{P})}
+4\,{ (\bar{\nabla} \csdot \IIo)} \csdot { \bar{\nabla} \bar{J}}
+3\,{ \IIo} \csdot { \bar{\nabla} \bar{\nabla} \bar{J}}\\
&+12\,{ \IIo} \csdot { \bar{\nabla} \bar{\nabla} \csdot \Fo}
+ 15 \, \Fo \csdot \bar{\Delta} \IIo
+{14\,{ (\bar{\nabla} \csdot \IIo)} \csdot { (\bar{\nabla} \csdot \Fo)}}
+{12}\,{ (\bar{\nabla} \IIo)} \csdot {( \bar{\nabla} \Fo)}
- 18 \, W_{abc \hat n}^\top \bn^a \Fo^{bc} \\
&-\tfrac{5}{6}\,{ \IIo} \csdot { \bar{\nabla} \bar{\nabla} K}
-{\tfrac{2}{3} {( \bar{\nabla} \csdot \IIo)} \csdot {{ \IIo} \csdot {( \bar{\nabla} \csdot \IIo)}}} 
- \tfrac{7}{4} (\bn \csdot \IIo) \csdot \bn K\\
&+24\, \IIo \csdot \bar{P} \csdot \Fo
-{9\,{ \bar{J}}\,{ \IIo} \csdot { \bar{P}}} 
-15\,{ \bar{J}}\,{ \IIo} \csdot { \Fo}
+{\tfrac{5}{3}\,K{ \IIo} \csdot { \bar{P}}} \\
&- 6 (\bar{\nabla} H) \csdot \big[
{\tfrac {1}{{8}}{\bar{\nabla} K}}
+{\tfrac {1}{{3}}{{\IIo \csdot ( \bar{\nabla} \csdot \IIo)} }}
 -\,{ \bar{\nabla} \csdot \Fo} \big]  \\
&+6\,H{ \bar{\nabla} \csdot \bar{\nabla} \csdot \Fo}
-{\tfrac{3}{2} {H{ \IIo} \csdot { \bar{\Delta} \IIo}}}
- 2 H (\bn \csdot \IIo)^2
-{\tfrac {3}{4}}\,H{ \bar{\Delta} K}
+12\,H{ \Fo} \csdot { \bar{P}}
+3\,HK{ \bar{J}} \\
&+{\tfrac {3}{{2}}{H{ \IIo}{}^{{a b}}{ \bar{\nabla}^c W_{{c a b \hat{n}}}^\top }}}
+{\tfrac {3}{{2}}{H{ \bar{W}}_{{a b c d}}{ \IIo}{}^{{a d}}{ \IIo}{}^{{b c}}}}
-12\,H \big[
{ \IIo} \csdot { C_{\hat n}^\top} 
+\tfrac{1}{2}\,{H}{{\rm{tr}} \IIo^3} 
-\,{H}{ \IIo} \csdot  \Fo \big]\in \Gamma(\ce \Sigma[-5])\, .
\end{align*}
\end{theorem}
\noindent
Since ${\mathcal B}_{\Sigma}$
is the variation of $E_{\rm GW}$,
 the leading term of $Ob$ demonstrates that 
$E_{\rm GW}$
satisfies the variational Property~(\ref{property}) of Definition~\ref{WillDef} and is therefore a higher Willmore energy. 
By inspection, the same holds 
for $E_{\rm Wm}$ and in turn all of the higher Willmore energies discussed above.
The proof of the above theorem is given in Section~\ref{prufs}. Note also that a formula for the~$d=5$ obstruction density in terms of a unit defining function, its jets, and a given choice of scale but not decomposed into hypersurface invariants, was presented in~\cite{hypersurface_old}.

\section{Conformally invariant powers of the Laplacian: Existence}

\label{doIcogito}

We now study the existence of  conformally invariant operators with no intrinsic counterparts on
generally conformally curved manifolds whose existence were established in~\cite{Will1}.

\subsection{Conformally invariant Laplacian squares  in two dimensions}

In dimensions $n\geq 3$, the Paneitz operator of Equation~\nn{P4-intr} expressed in terms of
the trace-free Ricci tensor and scalar curvature is given by
\begin{align*} 
 P_4  =\Delta^2 &- 
\tfrac{(n-1)^2-5}{2n(n-1)} \hh \nabla_a \circ Sc \circ \nabla^a
-\tfrac{(n+2)(n-2)(n-4)}{16n(n-1)^2} Sc^2
-\tfrac{n-4}{4(n-1)} \big((\Delta Sc)
+\tfrac{(n-2)(n+2)}{4n(n-1)}\hh Sc^2\Big)\\[2mm]
&+\tfrac1{n-2}\Big(4\nabla_a \circ \accentset{\,\circ}{Ric}{}^{ab}
\circ \nabla_b-\tfrac{n-4}{n-2}\accentset{\,\circ}{Ric}^2
\Big)
\, .
\end{align*}
Trouble in dimension two is signalled by the pole at $n=2$ in this
formula, although the residue vanishes because the trace-free Ricci
tensor is identically zero in this dimension. In fact, given only the
intrinsic data $(\Sigma^2,\cc_\Sigma)$, there is no natural
conformally invariant linear differential operator with leading terms
$\Delta^2$: It is not difficult to check that there is no choice of
parameters $\alpha$, $\beta$, $\gamma$ and $\delta$ such
that~$\Delta^2 +\alpha Sc \hh \Delta + \beta (\nabla_a Sc)\nabla^a +
\gamma (\Delta Sc) + \delta\hh Sc^2$ is an invariant operator acting
on weight one densities.

One way to view the failure of Paneitz in dimension two is to observe
that the numerator of the $n=2$ pole becomes conformally invariant (by
virtue of vanishing identically) in $n=2$ dimensions.  Notice,
however, that away from $n=2$ dimensions, $\mathring{Ric}/(n-2)$
transforms under conformal changes of metric $ \hat g = e^{2\omega} g$
by a shift of~$\nabla_{(a}\nabla_{b)\circ}\hh \omega +
(\nabla_{(a}\omega)\nabla_{b)\circ} \hh\omega$. In dimension two it is
still the case that such terms need to be cancelled.

To produce an invariant $P_4$ operator in two dimensions, additional data in the form of a  suitable tensor with the
transformation property $\nabla_{(a}\nabla_{b)\circ} \omega + (\nabla_{(a}\omega)\nabla_{b)\circ} \omega$
is necessary. The data of a tensor transforming this way is called a M\"obius structure and is also required to write down a tractor connection for two-dimensional conformal  manifolds~\cite{CalderbankCrelle}. A natural way to generate a M\"obius structure is by embedding the  2-manifold~$\Sigma$ in some ambient conformal manifold. The following result relies on this mechanism.

\begin{lemma}
\label{Horcrux}
Let $d=3$. Then
the mapping
\begin{align*}
f\mapsto
\bar{\Delta}^2 f &+ 4 \bar{\nabla}^a \circ \mathcal{P}_{ab} \circ \bar{\nabla}^b f \\
&+ \left[ 2 \bn \csdot \bn \csdot \mathcal{P} - \bar{\Delta} \mathcal{J} + 2 \mathcal{P}^2 - \mathcal{J}^2 + 2 \bn_a \left(\IIo^{ab} \bn \csdot \IIo_b \right) + 2 (\bn \csdot \IIo)^2 + \tfrac{1}{4} K^2 \right] f
\end{align*}
with $\mathcal{J} := \mathcal{P}_a^a$ where
\begin{align*}
\mathcal{P}_{ab} := P^\top_{ab} + H \IIo_{ab} + \tfrac{1}{2} H^2 \bar{g}_{ab} - \tfrac{1}{2} K \bar{g}_{ab}\,,
\end{align*}
defines a conformal squared Laplacian operator 
$$
\Gamma(\ce \Sigma[1])\to \Gamma(\ce \Sigma[-3])\, .
$$
\end{lemma}

\begin{proof} 
The claimed conformal variation can be verified by direct computation. Alternatively, there is a simple dimensional continuation argument. Consider, a $d-1$ dimensional hypersurface~$\Sigma$ embedded in a conformal manifold $(M,\cc)$ with $d\geq 4$. The operator 
\begin{equation}\label{thisoperator}
\hdb^A \circ P_2^{\Sigma \hookrightarrow M}\circ \hdb_A \end{equation}
maps $\Gamma(\ce \Sigma[\frac{5-d}2])\to \Gamma(\ce \Sigma[-\frac{3+d}2])$, where
$
P_2^{\Sigma \hookrightarrow M}
$
is defined in Theorem~\ref{Pk} and is given by
$$
P_2^{\Sigma \hookrightarrow M}
=\Delta^{\!\top} + \tfrac{3-d}2 \big(\bar J -\tfrac{1}{2(d-2)}\hh K\big)\, . 
$$
The operator~\nn{thisoperator} can be expressed in terms of the  Levi-Civita connection as
follows 
\begin{align*}
\hdb^A P_2^{\Sigma \hookrightarrow M} \hdb_A =\: & \tfrac{5-d}{2} \bar{\Delta}^2 \\
&+ \bn^a \circ \Big[ -2(d-5) \bar{P}_{ab} - 4(d-4) \Fo_{ab} - 2 \IIo^2_{ab} - \tfrac{(d-3)^2(d-7)}{4(d-1)(d-2)} K \bar{g}_{ab} \\
&\;\;\; \;\;\;\;\;\;\;\;\;\;\;\;\;+ \tfrac{1}{2} (d-3)(d-5)\bar{J} \bar{g}_{ab} \Big] \circ \bn^b  \\
&- \tfrac{d-5}{2} \Big( - \tfrac{d-5}{2} \bar{\Delta} \bar{J} - (d-5) \bar{P}^2 + \tfrac{1}{4} (d-1)(d-5) \bar{J}^2 + \tfrac{(d-3)^2}{4(d-1)(d-2)} \bar{\Delta} K\\
& - 2 (d-5) \Fo \csdot \bar{P} - (d-5) \Fo^2 - \tfrac{(d-3)^2 (d-5)}{4(d-1)(d-2)} K \bar{J} - \tfrac{d-5}{4(d-1)(d-2)^2} K^2 + 2 \bn \csdot \bn \Fo \\
&+ \tfrac{2}{d-2} \bar{\nabla}^a (\IIo_{ab} \bar{\nabla} \csdot \IIo^b) - \tfrac{d-5}{(d-2)^2} (\bn \csdot \IIo)^2 \Big) \, .
\end{align*}
As written, this identity cannot be dimensionally continued because neither $\bar{P}$ nor $\Fo$ are defined in hypersurface dimension $d-1 = 2$. Instead, observe that for all $d \geq 4$,
$$
\Fo = P^\top - \mathring{\bar{P}} + H \IIo - \tfrac{1}{d-1}\: \bar g\hh  \big(\bar{J} -  \tfrac{d-1}{2} H^2 + \tfrac{1}{2(d-2)} K)\quad\mbox{ and }\quad
\mathring{\bar{P}}= \mathring{\overline{Ric}}\hh  /(d-3)\, .$$
Note that $\mathring{\overline{Ric}}$ is well defined in all dimensions so can be ``dimensionally continued'' to hypersurface dimension $d-1 = 2$ (where it vanishes). Also $\bar{J}$ is defined in hypersurface dimension $d-1=2$ by the identity $\bar{J} = \overline{Sc}/(2(d-2))$. Thus, in these terms, we have
\begin{align*}
\hdb^A P_2^{\Sigma \hookrightarrow M} \hdb_A =&
\hspace{1mm} \tfrac{5-d}{2} \bar{\Delta}^2 
\\&+ \bar{\nabla}^a \circ \Big[ 4(4-d) P^\top_{ab} 
+ 2\mathring{\overline{Ric}}_{ab}
- 2 \IIo^2_{ab}
- \tfrac{d^2 - 12d+31}{4(d-2)} K \bar{g}_{ab} 
+ \tfrac{(d-3)^2}{2(d-1)} \bar{J} \bar{g}_{ab} \\&\hspace{13mm}
+ 4(4-d) H \IIo_{ab} 
+ 2(4-d) H^2 \bar{g}_{ab} \Big] \circ \bar{\nabla}^b \\
&- \tfrac{d-5}{2} \Big(
(5-d) (P^\top)^2 
+ \tfrac{(d-1)(d-5)}{4} \bar{J}^2 
- \tfrac{d-1}{2} \bar{\Delta} J 
+ 2 \bar{\nabla} \csdot \bar{\nabla} \csdot P^\top \\
&+ \tfrac{(d-1)(d-5)}{4} H^4
 - \tfrac{(2d-3)(d-5)}{2(d-2)} H^2 K 
- (d\!-\!5) H^2 \bar{J} 
- 2(d\!-\!5) H \IIo \csdot P^\top 
+ 2 H \bar{\nabla} \csdot \bar{\nabla} \csdot \IIo \\&
+ 2 H \bar{\Delta} H 
- \tfrac{(d-5)^2}{4(d-2)} K \bar{J} 
+2 \IIo \csdot \bar{\nabla} \bar{\nabla} H 
+4 (\bar{\nabla} \csdot \IIo) \csdot \bar{\nabla} H 
- \tfrac{(d-5)}{(d-2)^2} (\bar{\nabla} \csdot \IIo)^2 \\&
+ \tfrac{(d-5)}{4(d-2)} \bar{\Delta} K 
+2 (\bar{\nabla} H)^2 
+ \tfrac{2}{d-2}\bar{\nabla}^a (\IIo_{ab} \bar{\nabla} \csdot \IIo^b) \Big)\,.
\end{align*}
Taking the limit where $d \rightarrow 3$ so that $\mathring{\overline{Ric}} = 0$ and defining $\mathcal{P} := P^\top + H \IIo+ \tfrac{1}{2} H^2 \bar{g} - \tfrac{1}{2} K \bar{g}$ completes the proof.
\end{proof}

\subsection{Conformally invariant Laplacian cubes  in four dimensions}

The operator 
\begin{equation}\label{throwback}
\hd^{T A}\circ
P_4^{\sss\Sigma\hookrightarrow M}
\circ
 \hd^T_{\hh A}
\end{equation}
defines a mapping
$$
\Gamma\big(\ce \Sigma\big[\tfrac{7-d}2\big]\big)
\longrightarrow
\Gamma\big(\ce \Sigma\big[\tfrac{-5-d}2\big]\big)\, ,
$$
with leading derivative term proportional to $\bar \Delta^3$;
see~\cite{Will2}. When $d=5$ (so $\Sigma$ is a four manifold) this defines a sixth order Laplacian power~$P_6^{\Sigma\hookrightarrow M}$.
An explicit Riemannian formula for $P_6^{\Sigma\hookrightarrow M}$ follows as a direct corollary of Theorem~\ref{Big_Formaggio}, however this necessarily involves many terms, see for example already the result of~\cite{GOpet} for the intrinsic sixth order Laplacian power of Graham, Jennes, Mason, and Sparling (GJMS)~\cite{GJMS} for conformal manifolds of dimension five and higher.  However, when $\Sigma$ is embedded as the conformal infinity of a Poincar\'e--Einstein structure it is possible to write down a relatively compact formula. Note that there is no contradiction with an old result of Graham that no 
sixth order GJMS operator exists on generally curved conformal four manifolds~\cite{GraNon}. 
That result relies on the fact that the Bach tensor intrinsic to a four manifold is a conformal invariant, indeed it appears as the residue of a $1/(d-5)$ pole in the intrinsic result of~\cite{GOpet}.
However, the operator $P_6^{\Sigma\hookrightarrow M}$ above
probes data not fixed by the intrinsic conformal geometry of $\Sigma$ which allows for the replacement of the intrinsic Bach tensor by an object with the correct transformation property:
\begin{theorem}\label{Ronkonkoma}
Let 
$\Sigma$ be embedded as the conformal infinity of a Poincar\'e--Einstein structure
~$(M,\cc)$.
Then the
 mapping
\begin{align*}
f\mapsto 
\bar{\Delta}^3 f 
&- 3 \bar{J} \bar{\Delta}^2 f 
+ 16 \bar{P} \csdot \bn \bn \bar{\Delta} f 
+10 (\bn \bar{J}) \csdot \bn \bar{\Delta} f 
+ 16 (\bn \bar{P}) \csdot (\bn \bn \bn f) \\
& +\Big((\bar{\Delta}\bar{J})\hh \bar{g}
+ 20 (\bn \bn \bar{J})
 -\bar{J}^2 \bar{g}
- 16 \bar{J} \bar{P}
- 24 \bar{P}^2 \bar{g}
+ 32\bar{W}(\pdot, \bar{P}, \pdot)
+ 144 \bar{P}^2
+ 16 B^\top
\Big) \csdot  \bn \bn f \\
&+ \Big( 8 (\bn \bar{\Delta} \bar{J})
- 7 (\bn \bar{J}^2)
+ 72  (\bn \bar{J}) \csdot \bar{P}
+ 32 (\bn \bar{P}^2)
- 80 \bar{C}(\pdot, \bar{P})
+ 16 (\bn \csdot B^\top)
\Big) \csdot \bn f \\
&+ \Big( (\bar{\Delta}^2 \bar{J})
+ 3 \bar{J}^3
- 24 \bar{J} \bar{P}^2
-5 \bar{J} (\bar{\Delta} \bar{J})
+ 8 \bar{P} \csdot \bar{W}(\pdot, \bar{P}, \pdot)
+ 48 \bar{P}^3
+16 \bar{P} \csdot (\bn \bn \bar{J}) \\
&\phantom{+\Big(}\;+ 8 (\bn \bar{P})^2
-4 \bar{C}^2
+ 2 (\bn \bar{J})^2
- 16 \bar{P} \csdot (\bar{B} - B^\top)
+ 8 (\bn \csdot \bn \csdot B^\top)
\Big) f \, ,
\end{align*}
where $B^\top$ is the projection of the Bach tensor of $(M,g)$ to the hypersurface $\Sigma$, defines a conformal cubed-Laplacian operator
$$
\Gamma(\ce \Sigma[1])\to \Gamma(\ce \Sigma[-5])\, .
$$
\end{theorem}

\noindent
The proof is given in  Section~\ref{prufs}.
Note that the conclusion of the theorem also holds for
any asymptotically Poincar\'e--Einstein structure, meaning 
any metric $g^o$ on $M\backslash \Sigma$ whose trace-free Schouten tensor obeys $\mathring P={\mathcal O}(\sigma^2)$.

Following the discussion above, the appearance of the projected bulk Bach tensor $B^\top$ in Theorem~\ref{Ronkonkoma} 
parallels that of $P^\top$ in Lemma~\ref{Horcrux}.
This is part of a more general picture linked to the conformal fundamental forms of~\cite{BGW}.
Just as a third fundamental form, in the guise of the Fialkow tensor, was used to construct
a tensor of the form $P^\top+\cdots$ with the same transformation properties $\bar P$ in dimensions $d\geq 4$, the (conditional) higher fundamental forms of~\cite{BGW} can be used for the same purpose. 
In particular, in dimension $d= 6$, a conditional fifth fundamental form is given by
$$
\mathring B^\top -\tfrac{d-4}{d-5} \bar B+\cdots
 $$
and is invariant, and hence can be used to construct a tensor with the same transformation properties as~$\bar B$ (when $d\geq 6$). A dimensional continuation argument can then be used to extract the tensor in $d=5$ dimensions required in the above theorem.

\section{Proofs}\label{prufs}

\color{black}

In this section we prove Theorem~\ref{Big_Formaggio} as well as some consequences that we recorded in  Results~\ref{forIamamerecorollary}, \ref{P4N},\ref{dots}, \ref{Voldemort}, and~\ref{Ronkonkoma}.
Our methods rely heavily on hypersurface tractor calculus, in particular the general theory presented in~\cite{Will1,Will2,Staff,Grant} as well as its further development given in~\cite{BGW}. We therefore begin 
by recording a collection of useful tractor identities.


\subsection{Consequences of Leibniz's failure}\label{Flack}

In this section we develop identities involving any weight one density~$\sigma$ and its interactions with the Thomas-$D$ operator and the canonical tractor~$X$. These all follow from straightforward applications of the \textit{Leibniz failure} identity of~\cite{Taronna}, proved in~\cite{Will1}, which encodes the failure of the  Thomas-$D$ operator  to be a derivation:
\begin{proposition}\label{leib-failure}
Let~$T_i \in \Gamma(\ce^\Phi M[w_i])$ for~$i = 1,2$, and~$h_i := d+2w_i$,~$h_{12} := d+2w_1 + 2w_2 - 2$ with~$h_i \neq 0 \neq h_{12}$. Then,
$$\hd^A (T_1 T_2) - (\hd^A T_1) T_2 - T_1 (\hd^A T_2) = -\frac{2}{h_{12}} X^A (\hd_B T_1) (\hd^B T_2)\,.~$$
\end{proposition}
 
In principle, all of the following identities can be proved by hand using the above proposition. We have employed   the symbolic algebra system FORM~\cite{Jos} to handle  more intricate cases. 
First we need some simplifying notations. 
\begin{definition}\label{tractor-defs}
Let~$\sigma \in \Gamma(\ce M[1])$. Then, in dimensions $d$ where the right hand sides below are well-posed, we define the following quantities:
\begin{equation*}
\begin{aligned}[c]
I^{\sigma}_A\:  &:= \hd_A \sigma \in \Gamma(\ct M[0]) \, ,\\
P^\sigma_{AB} &:= \hd_A I^\sigma_B \in \Gamma(\odot^2_{\circ} \ct M[-1])\, , \\
\dot{P}^{\sigma}_{AB}  &:= I^\sigma \csdot \hd\, P^\sigma_{AB} \in \Gamma(\odot^2_{\circ} \ct M[-2])\, ,\\
\ddot{P}^{\sigma}_{AB}  &:= I^\sigma  \csdot \hd^2\, P^\sigma_{AB} \in \Gamma(\odot^2_{\circ} \ct M[-3])\, ,
\end{aligned}
\qquad \qquad
\begin{aligned}[c]
K^\sigma_{\rm e} &:= P^\sigma_{AB} P^{\sigma AB} \in \Gamma(\ce M[-2])\, ,\\
\dot{K}^\sigma_{\rm e} &:=  I^\sigma \csdot \hd\, P^\sigma_{AB} P^{\sigma AB} \in \Gamma(\ce M[-3])\, ,\\
\ddot{K}^\sigma_{\rm e} &:= (I^\sigma  \csdot \hd)^2\, P^\sigma_{AB} P^{\sigma AB} \in \Gamma(\ce M[-4])\, ,\\
\dddot{K}^\sigma_{\rm e} &:= (I^\sigma  \csdot \hd)^3 \, P^\sigma_{AB} P^{\sigma AB} \in \Gamma(\ce M[-5])\, .
\end{aligned}
\end{equation*}
\end{definition}

We are particularly interested in the case that $\sigma$ is an asymptotic unit defining density 
determined by a conformal hypersurface embedding $\Sigma\hookrightarrow (M,\cc)$.
In that case we shall often drop the superscript~$\sigma$. Note that the densities $K_{\rm e}$, $\dot{K}_{\rm e}$,...  defined above agree with the definitions for the same given in Section~\ref{NinaSimone}.

Many of the identities that follow only hold when when evaluated along~$\Sigma$. In order to establish equality along~$\Sigma$, some of the results below require the tangential  property of the {\it tangential Thomas-$D$ operator}~$\hd^T$ defined below.
\begin{proposition} \label{Dt-trans}
Let ~$w + \tfrac{d}{2} \neq 1, \tfrac{3}{2},2$, let~$\sigma$ be an asymptotic unit defining density for~$\Sigma \hookrightarrow (M^d, \cc)$, and let~$I = \hd \sigma$. Then, the operator
\begin{align} \label{DT}
\hd^T:=\hd_A - I_A I \csdot \hd +  \frac{X_A}{d+2w-3} I \csdot \hd^2\, ,
\end{align}
mapping $ \Gamma(\ct^\Phi M[w]) \rightarrow \Gamma(\ct M \otimes \ct^\Phi M[w-1])$,
is tangential. Moreover, if~${\mathcal W}^A$ is any operator acting on tractors  of weight~$\frac{1-d}2$ that obeys
~$${\mathcal W}^A\circ X_A=0\, ,$$ then the operator
$$
{\mathcal W}\csdot \hat D^T :=
{\mathcal W}^A\circ \big(\hd_A - I_A I \csdot \hd \big)
$$
is also tangential.
\end{proposition}

\begin{proof}
The proof can be found in~\cite{Forms, Will2, BGW}.
\end{proof}


Just as in Riemannian geometry, where the Gauss formula~\nn{GF} relates the ambient connection  to the hypersurface connection,  a similar result relates the  bulk and hypersurface Thomas-$D$ operators; see~\cite{Will2,BGW}.
For this we need a tractor $\Gamma$ that
combines the Fialkow tractor and tractor second fundamental form:
\begin{align}\label{Gamma}
\Gamma_{ABC} := 2 N_{[C} L_{B]A} + 2 X_{[C} F_{B]A} + \tfrac{K}{(d-1)({d}-2)} X_{[C} \bar h_{B]A}.
\end{align}
In these terms we have the following result:
\begin{theorem}[Gau\ss--Thomas formula~\cite{BGW}, Theorem 1.2]\label{thesecondhorcrux}
Let~$w+\frac d2 \neq 1,\frac 32,2$. 
Acting  on weight~$w$ tractors, the bulk tangential and hypersurface Thomas-$D$ operators obey
$$
\hd^T_A\stackrel\Sigma =  \hdb_A +\Gamma_A {}^\sharp - \tfrac{X_A}{d + 2w-3} \bigg{\{} 
2\Gamma_B {}^\sharp \circ \hdb^B + \Gamma^B{}^\sharp \circ \Gamma_B{}^\sharp 
+ \tfrac{1}{(d-1)(d-2)} \left[\left(\hdb_{[\pdot} K \right)  \hh X_{\pdot ]} \right]^\sharp - \tfrac{(3d-1)wK}{2(d-1)(d-2)} \bigg{\}}\, .
$$
\end{theorem}

\medskip

\subsubsection{Commutators}
The simplest  identities we require are  commutators involving the objects in Definition~\ref{tractor-defs} and the Thomas-$D$ operator.

\begin{lemma}\label{Dcomm}
Let~$\sigma$ be any weight~$w=1$ density.
Acting on tractors of weight~$w$ such that~$\hd$ as it appears below is well-defined, the following operator identities hold:
\begin{align*}
\big[\hd_A, \sigma\big] \;\; &= I_A -  \tfrac 2h X_A I\csdot \hd\, ,\\
\big[\hd_A, X_B \big] &= h_{AB} - \tfrac2h X_A \hd_B  \, ,\\
\big[\hd _A, I_B\big] \, &= P_{AB} -\tfrac2{h-2} X_A P_{CB}\hd ^C\, ,
\end{align*}
where~$h := d+2w$.
\end{lemma}

\begin{proof}
The results follow from a direct application of Proposition~\ref{leib-failure} and Definition~\ref{tractor-defs}.
\end{proof}

\begin{remark}
It is useful to define the operator ${\sf h}$ which, when acting on a weight $w$ tractor $T$, returns the value $h=d+2w$, {\it i.e.}, 
$$
{\sf h} \hh T = h T\, .
$$
Note that rational functions of the operator ${\sf h}$ are defined in the obvious way.
\end{remark}

\begin{lemma} \label{IDcomm}
Let~$\sigma$ be an asymptotic unit defining density. Then,
acting on tractors with weight~$w=:\frac12 (h-d)$ such that~$\hd$ as it appears below is well-defined, the following operator identities hold: %
\begin{align} 
\left[I \csdot \hd, X^A \right] &= I^A - \tfrac{2 \sigma}h\,  \hd^A \, ,\label{IDcommaX} \\[2mm]
\left[ I \csdot \hd, \hd_A \right] &= -P_{AB} \hd^B - I^B \left[\hd_A, \hd_B \right] + \tfrac{2}{h-4} \, X_A P^{BC} \hd_B \hd_C \, ,\label{IDcommD} \\[2mm]
\IdD^k\circ  \sigma &= \tfrac{h-2k}{h}\, \sigma \IdD^k  +\tfrac{k(h-k+1)}{h} \, \IdD^{k-1}+{\mathcal O}(\sigma^{d-k+1})\, , \\[2mm]
\left[I \csdot \hd , I_A \right] &= \tfrac1{d-2}\, {X_A K} - \tfrac{2\sigma}{h-2} P_{AB} \hd ^B +{\mathcal O}(\sigma^{d-2}) \label{ID,I}\, ,\\[2mm]
\left[I \csdot \hd, P^{AB} \right] &= \dot{P}^{AB} - \tfrac{2 \sigma}{h-4}\, \big(\hd^E P^{AB} \big) \hd_E \,.  \label{IDcommaP}
\end{align}
\end{lemma}
\begin{proof}
As in Lemma~\ref{Dcomm}, the lemma follows from Proposition~\ref{leib-failure} and Definition~\ref{tractor-defs}. The third identity requires a simple induction argument. The third and fourth identities also require that $I^2 = 1 + {\mathcal O}(\sigma^d)$. Also, Equation~\nn{ID,I}
uses that $\hat D_A I_B=\hat D_B I_A$.
Note that Equations~(\ref{IDcommaX}),~ (\ref{IDcommD}) and~\nn{IDcommaP} in fact hold for any weight
one density $\sigma$.
\end{proof}

It is also useful to record an identity relating the commutator of two Thomas-$D$ operators and the~$W$-tractor.
\begin{proposition} \label{DD-comm}
In dimension~$d\neq 4$, the commutator of Thomas-$D$ operators is given by
\begin{align*}
\left[D_A, D_B \right] = (d+2w-4)(d+2w-2)W_{AB}{}^\hash +4 X_{[A} W_{B]C}{}^\hash\circ  D^C.
\end{align*} 
\end{proposition}
\begin{proof}
The proof is given in~\cite{BGW,GOmin,GOadv}.
\end{proof}

The above proposition inspires one more useful commutator result:
\begin{lemma} \label{DWcomm}
Let~$\sigma$ be an asymptotic unit defining density and let~$T^{\Phi AB} \in \Gamma(\ct^\Phi M \otimes \wedge^2 \ct M[w'])$ such that~$T^{\Phi \sharp}$ denotes the tractor-valued tractor-endomorphism~$T^{\Phi \sharp} V^A \mapsto T^{\Phi A}{}_B V^B$, that extends by Leibniz to act on general tractor tensors. Then, acting on a tractor of weight~$w$ such that~$\hd$ as it appears below is well-defined, the following operator identities hold:
\begin{align*}
\big[\hd_A, T^{\Phi \sharp} \big] \hh &= (\hd_A T^{\Phi \sharp}) - T^\Phi{}_A{}^B \hd_B - \tfrac{2}{h+2w'-2} X_A (\hd_C T^{\Phi \sharp}) \circ \hd^C + \tfrac{2}{h+2w'-2} X_A (\hd_C T^{\Phi C}{}_D) \hd^D\,, \\
\big[I \csdot \hd, T^{\Phi \sharp} \big] &= (I \csdot \hd T^{\Phi \sharp}) - \tfrac{2 \sigma }{h+2w'-2} (\hd_C T^{\Phi \sharp}) \circ \hd^C + \tfrac{2 \sigma}{h+2w'-2} (\hd_C T^{\Phi C}{}_D) \hd^D\,,
\end{align*}
where~$h:= d+2w$.
\end{lemma}
\begin{proof}
This operator identity is an elementary application of Lemma~\ref{leib-failure} while accounting for the action of~$T^{\Phi\hash}$.
\end{proof}

%

\subsubsection{Operator identities along~$\Sigma$}
Here we provide a list of  operator identities valid along~$\Sigma$. These identities were proved using Definition~\ref{tractor-defs}, Lemmas~\ref{Dcomm} and~\ref{IDcomm}, and 
the computer algebra system FORM. 
The next lemma shows how to convert various operators involving Thomas-$D$ operators to combinations of their tangential parts and Laplace--Robin operators. 
\allowdisplaybreaks

\begin{lemma} \label{P-operator-ids}
Acting on tractors of weight~$w=:\frac12 (h-d)$, and such that~$\hd$ and $\hd^T$ as they appear below are well-defined,
the following operator identities hold:
\begin{align}
P^{AB} \circ \hd_A \circ \hd_B &\eqSig P^{AB} \hd^T_A \hd^T_B + \tfrac{h-4}{d-2} K I \csdot \hd \, ,\\[4mm] 
P^{AB} \circ \hd_A \circ P_B^C \circ \hd_C &\eqSig P^A_C P^{CB} \hd^T_A  \hd^T_B + P_{AB} (\hd^A P^{BC}) \hd^T_C \, ,\\[4mm] 
P^A_C P^{CB} \circ \hd_A \circ \hd_B &\eqSig -\tfrac{1}{h-3} K I \csdot \hd^2 + P^3 I \csdot \hd + P^A_C P^{CB} \hd^T_A \hd^T_B \, ,\\[4mm] 
(\hd^A K) \circ \hd_A &\eqSig \tfrac{2}{h-3} K I \csdot \hd^2 + \Kd I \csdot \hd + (\hd^A K) \hd^T_A \, ,\\[4mm] 
\begin{split}
I_A (\hd^C P^{AB}) \circ \hd_C \circ \hd_B &\eqSig \tfrac{2h - d - 6}{(d-2)(h-3)} K I \csdot \hd^2 + \tfrac{(d-6)(h-d-2)}{2(d-2)(d-4)} \Kd I \csdot \hd\\[2mm]
&\quad - \tfrac{h-6}{d-4} P^3 I \csdot \hd - P^A_C P^{CB} \hd^T_A  \hd^T_B \\[2mm]
&\quad+ \tfrac{h-d-2}{d-4} P_{AB} (\hd^A P^{BC}) \hd^T_C + \tfrac{(d-6)(h-d-2)}{2(d-2)(d-4)}
(\hd^A K) \hd^T_A\, ,
\end{split} \\[4mm] 
\begin{split}
I_A \ddot{P}^{AB} \circ \hd_B &\eqSig \tfrac{2}{(d-2)(h-3)} K I \csdot \hd^2 + \tfrac{2(d-5)}{(d-2)(d-4)} \Kd I \csdot \hd - \tfrac{2}{d-4} P^3 I \csdot \hd \\[2mm]
&\quad + \tfrac{2}{d-4} P_{AB} (\hd^A P^{BC}) \hd^T_C - \tfrac{2}{(d-2)(d-4)} (\hd^A K) \hd^T_A \\[2mm]
&\quad + \tfrac{h-d}{(d-2)^2} K^2 + \tfrac{h-d}{2(d-2)} \Kdd\, ,
\end{split}\\[4mm]
\begin{split}
(\hd_A \dot{P}^{AB}) \circ \hd_B &\eqSig \tfrac{d-2}{(d-6)(h-3)} K I \csdot \hd^2 + \tfrac{d-4}{2(d-6)} \Kd I \csdot \hd \\[2mm]
&\quad  - \tfrac{2}{d-6} P^3 I \csdot \hd + (\hd_A \dot{P}^{AB}) \hd^T_B\, ,
\end{split}\\[4mm]
\begin{split}
\hd_A \circ \dot{P}^{AB} \circ \hd_B &\eqSig \tfrac{(d+2)h - 8d}{(d-2)(h-3)(h-8)} K I \csdot \hd^2 + \tfrac{(d^2 - 4d-4)h - 8d^2 + 44d - 24}{2(d-2)(d-4)(h-8)} \Kd I \csdot \hd \\[2mm]
&\quad - \tfrac{2(h-6)}{(d-4)(h-8)} P^3 I \csdot \hd  + \tfrac{h-6}{h-8} \dot{P}^{AB} \hd_A \hd_B \\[2mm]
&\quad + \tfrac{(d-6)(h-d-2)}{(d-2)(d-4)(h-8)} (\hd^A K) \hd^T_A + \tfrac{(d-6)(h-6)}{(d-4)(h-8)} (\hd_A \dot{P}^{AB})\hd^T_B\,. 
\end{split}
\end{align}
\end{lemma}
\begin{proof}
These identities  result from a series of symbolic algebra calculations using FORM. The files performing these calculations are documented in~\cite{FormFiles}.
\end{proof}

In addition to Definition~\ref{tractor-defs} and Lemmas~\ref{Dcomm} and~\ref{IDcomm}, to establish the following lemma for operators involving the Laplace--Robin operator, we also need to use Lemma~\ref{P-operator-ids}. 
\begin{lemma} \label{TIDU}
Acting on tractors of weight~$w=:\frac12 (h-d)$, and such 
that~$\hd$ and $\hd^T$ as they appear below are well-defined,
the following operator identities hold:
\begin{align} 
\phantom{help}&\nonumber\\
X^A  \circ I\csdot \hd \circ X_A &\eqSig   0 \label{X1X} \, ,\\[2mm] 
X^A  \circ I\csdot \hd^2 \circ X_A &\eqSig   -\tfrac{h-d}h \label{X2X} \, ,\\ 
X^A \circ I\csdot \hd^3 \circ X_A &\eqSig -  \tfrac{3(h-d-2)}{h} I\csdot\hd  \label{X3X} \, ,\\[2mm]
X^A \circ I\csdot \hd^4 \circ X_A &\eqSig - \tfrac{6(h-d-4)}{h} I\csdot \hd^2
-  \tfrac{2h^3 - (2d+20)h^2 + (16d+56)h -48d}{(d-2)h(h-2)(h-4)}\,  K 
\, ,\label{X4X}  \\[4mm]
X^A \circ I\csdot \hd \circ I_A &\eqSig 0 
\label{X1I} \, ,\\[2mm] 
X^A \circ I\csdot \hd^2 \circ I_A &\eqSig 0 \label{X2I} \, ,\\[2mm] 
X^A \circ I\csdot \hd^3 \circ I_A &\eqSig  \tfrac{h^2 -(d+4)h + 8d-8}{(d-2)(h-2)(h-4)} K \label{X3I} \, ,\\[2mm] 
\begin{split}
X^A \circ I\csdot \hd^4 \circ I_A &\eqSig\, 
    \tfrac{4(h-5)(h^2 - (d+8)h + 10d +4)}{(d-2)(h-2)(h-4)(h-6)}\, K \, I\csdot \hd  - 
\tfrac{12}{(h-2)(h-4)} P^{AB} \hd^T_A \hd^T_B  \\& \quad + \tfrac{3h^2 - (3d+12)h + 30d-36}{(d-2)(h-2)(h-4)} \dot{K}\, ,
\end{split} \label{X4I} \\[2mm] 
I^A\circ  I\csdot \hd \circ X_A &\eqSig 1 \label{I1X} \, ,\\[2mm] 
I^A \circ  I\csdot \hd^2 \circ X_A &\eqSig  \tfrac{2(h-1)}{h}  I\csdot \hd \label{I2X} \, ,\\[2mm] 
I^A  \circ  I\csdot \hd^3 \circ X_A &\eqSig   \tfrac{3(h-2)}{h} \, I\csdot \hd^2 +  \tfrac{2h^2 -(d+8)h+6d}{h(d-2)(h-2)}\,  K \label{I3X} \, ,\\[2mm] 
\begin{split}
I^A \circ I \csdot \hd^4 \circ X_A &\eqSig \tfrac{4(h-3)}{h} I \csdot \hd^3 - \tfrac{4(h-3)(-2h^2 + (d+16)h - 8d-24)}{(d-2)h(h-2)(h-4)} K I \csdot \hd \\[2mm]
&\quad - \tfrac{12}{h(h-2)} P^{AB} \hd^T_A \hd^T_B + \tfrac{3h^3 - (d+28)h^2 + (14d + 68)h - 48d}{(d-2)h(h-2)(h-4)}\dot K\, ,
\end{split} \label{I4X}\\[4mm] 
X^A \circ I \csdot \hd \circ  \hd_A &\eqSig \tfrac{h-d-2}{2} I \csdot \hd \label{X1D} \, ,\\[2mm] 
X^A \circ I \csdot \hd^2 \circ  \hd_A &\eqSig \tfrac{h-d-4}{2} I \csdot \hd^2 + \tfrac{h-d}{2(d-2)}K \label{X2D} \, ,\\[2mm] 
\begin{split}
X^A \circ I \csdot \hd^3 \circ \hd_A &\eqSig \phantom{+} \tfrac{h-d-6}{2} I \csdot \hd^3- \tfrac{2(2h-d-10)}{(h-4)(h-6)}  P^{AB} \hd^T_A \hd^T_B   \\& \quad + \tfrac{3h^2 - (3d+28)h + 22d+52}{2(d-2)(h-6)} K I \csdot \hd  + \tfrac{h-d}{d-2} \dot{K} + \tfrac{2}{h-4} I^A \hd^B [\hd_A, \hd_B]\, ,
\end{split} \label{X3D} \\[2mm]
\begin{split}
X^A \circ I \csdot \hd^4 \circ   \hd_A  &\eqSig \,\,  \tfrac{h-d-8}{2} I \csdot \hd^4 \\
& \quad+  \resizebox{0.6 \textwidth}{!} {$\tfrac{3h^5 - (3d+83)h^4 + (71d+908)h^3 - (608d+4932)h^2 + (6d^2+2244d+13224)h - 48d^2 - 3024d-13536}{(d-2)(h-3)(h-4)(h-6)(h-8)}~$}K I \csdot \hd^2 \\
& \quad+\tfrac{h-d}{2(d-2)^2} K^2  
+\tfrac{3(h-d)}{2(d-2)}\ddot{K}\\ 
& \quad+ \tfrac{4h^3 -(4d+61)h^2 + (58d+246)h - 210d-108}{(d-2)(h-6)(h-8)}  \dot{K} I \csdot \hd  \\
& \quad+ \tfrac{6}{h-4} I^A \hd^B I \csdot \hd [\hd_A, \hd_B]
+ \tfrac{2}{h-6} I^A \hd^B [\hd_A, \hd_B] I \csdot \hd \\
& \quad- \tfrac{4(h-7)}{(h-4)(h-6)} I^A [\hd_A, \hd_B] I_C [\hd^B, \hd^C] \\
& \quad+ \tfrac{2(10h^2 - (3d+130)h + 24d + 408)}{(h-4)(h-6)(h-8)} P^{AB} I^C \hd_A [\hd_B, \hd_C] \\
& \quad+\tfrac{2(5h^2 - (3d+74)h + 24d + 264)}{(h-4)(h-6)(h-8)} \Big[ P^3 I \csdot \hd + P^{AB} I^C   [\hd_A, \hd_C] \left( \hd^T_B + I_B I \csdot \hd  \right) \Big] \\
& \quad-  \tfrac{4(4h^2 -(2d+53)h +14d+172}{(h-4)(h-6)(h-8)} P^{AB} \hd^T_A \hd^T_B I \csdot \hd \\
& \quad+ \tfrac{2(13h^2 -(6d+184)h+48d+624)}{(h-4)(h-6)(h-8)}  P^A_C P^{CB} \hd^T_A \hd^T_B  \\
& \quad+ \tfrac{2((8d-30)h^2 - (3d^2+98d-420)h + 24d^2 + 264d - 1392)}{(d-4)(h-4)(h-6)(h-8)}  P^{AB} \left(\hd_A P_B^C \right) \hd^T_C  \\
& \quad- \tfrac{3(h-d)}{(d-2)(h-6)} \left(\hd^A K \right) \hd^T_A \\
& \quad-\tfrac{6(2h-d-12)}{(h-4)(h-6)}  \dot{P}^{AB} \hd_A \hd_B \\
& \quad - \tfrac{2(d-6)}{(d-4)(h-8)} \left(\hd_A \dot{P}^{AB} \right) \hd^T_B\, ,
\end{split} \label{X4D} \\[4mm]
\hd_A \circ I \csdot \hd^2 \circ X^A &\eqSig \tfrac{(h+d)(h+2)}{2h} I \csdot \hd^2  + \tfrac{h^3 +(d-6)h^2 -2dh+8d}{2(d-2)h(h-4)} K  \, ,\label{D2X} \\[4mm]
\begin{split}
\hd_A \circ I \csdot \hd^2\circ   \hd^A &\eqSig -  \tfrac{h-6}{2(h-8)}  \dot{K} I \csdot \hd \\
& \quad +  \tfrac{2h^3 - 32h^2 + (d+160)h-8d-236}{(h-3)(h-6)(h-8)} K I \csdot \hd^2 \\
& \quad + I^A \hd^B \left(I \csdot \hd [\hd_A, \hd_B] + [\hd_A, \hd_B] I \csdot \hd \right) \\
& \quad - \tfrac{d-2}{h-6} P^{AB} I^C \hd_A [\hd_B, \hd_C] \\
& \quad- \tfrac{h^2+(d-14)h -8d+52}{(h-6)(h-8)}  \left(P^3 I \csdot \hd + P^{AB} I^C [\hd_A, \hd_C] \hd_B \right) \\
& \quad+   \tfrac{2(d-2)(h-7)}{(h-6)(h-8)}  P^{AB} \hd^T_A \hd^T_B I \csdot \hd  \\
& \quad- \tfrac{h^2 + (2d-14)h - 16d+56}{(h-6)(h-8)}P^A_C P^{CB} \hd^T_A \hd^T_B  \\
& \quad+ \tfrac{2h^2 - (d^2-4d+24)h +8d^2 -36d+88}{(d-4)(h-6)(h-8)} P^{AB} \left(\hd_A P_B^C \right) \hd^T_C  \\
& \quad+   \tfrac{d-2}{h-6} \dot{P}^{AB} \hd_A \hd_B \\
& \quad- \tfrac{(d-6)(h-6)}{(d-4)(h-8)}  \left(\hd_A \dot{P}^{AB} \right) \hd^T_B \,.
\end{split}  \label{D2D}
\end{align}
\end{lemma}

\begin{proof}
This lemma was proved sequentially using FORM---generally, the more complex identities rely on the less complex ones.
\end{proof}

We are now ready to tackle the proof of the central result  given in Theorem~\ref{Big_Formaggio}].

\color{black}
\subsection{Boundary tractor formula for the extrinsic Paneitz operator}

Recall that the Paneitz operator~$P_4$ intrinsic to 
a conformal~$n$-manifold~$(\Sigma,\cc_\Sigma)$, acting on weight~$2-\frac n2\neq 0$ densities, can be expressed as
$$
P_4=\tfrac8{n-4}\,  \hat{D}^A\circ  P_2 \circ \hat {D}_A\, ,
$$
where~$P_2$ is the {\it Yamabe operator/conformal Laplacian} on weight $1-\frac n2$ tractors defined by $D^A T=-X_A P_2 T$ for $T\in \Gamma(\ct M[1-\frac n2])$; see for example~\cite{GOpar,GOpet}.
Therefore, to write the holographic formula of Theorem~\ref{Pk} for~$\PanE$  explicitly in terms of  hypersurface data, our plan is to convert the operator~$(I\csdot \hd)^4$ to the form
$$\PanE \stackrel\Sigma= \tfrac{8}{d-5} \hh \hat{ D}^{T\, A} \circ P_2^{\sss\Sigma\hookrightarrow M} \circ \hat{ D}^T_A + \text{lower derivative terms.}$$
Here the {\it extrinsically-coupled Yamabe operator}~$P_2^{\sss\Sigma\hookrightarrow M}$ is defined on tractors of weight~$\frac{3-d}2$ by the tangential operator~$I\csdot \hd^2$ (again see Theorem~\ref{Pk}). In the case that the operator $\PanE$  acts on sections of $\ct^\Phi \Sigma\big[\frac{5-d}2\big]$, Theorem~\ref{thesecondhorcrux}
can be used to convert tangential Thomas-$D$ operators to hypersurface ones. The result of this computation is given below.

\begin{theorem}\label{Big_Formaggio}
Let~$\sigma$ be a unit conformal defining density for~$\Sigma \hookrightarrow (M, \cc)$,where~$M$ has dimension~$d  \geq 8$.
Then, acting on sections of  $\Gamma(\ct^\Phi M[\tfrac{5-d}{2}]|_{\Sigma})$,
\begin{align*}
\PanE  \stackrel\Sigma=& \hh  \hh \tfrac{8}{d-5}\hh \hd^{T A}\circ
P_2^{\sss\Sigma\hookrightarrow M}
\circ
 \hd^T_{ A} \\
&+\Big(\tfrac{4(2d-11)}{d-5} L^{AC} L_C^B  
+ \tfrac{4(d-3)(d-6)}{d-5} F^{AB} 
- \tfrac{8}{d-5} W^{AB\, \sharp} \Big) \circ \hd^T_{A}\circ \hd^T_{B} \\
&+\Big( \tfrac{4(2d-11)}{d-5} L_{BC} (\hdb^B L^{CA}) 
- \tfrac{2(d^2 - 8d+17)}{(d-1)(d-2)(d-5)} (\hdb^A K) 
+ \tfrac{4(3d-16)}{d-5} N_B L_{CD} W^{BCDA} \\&\phantom{+\Big(Q}
- \tfrac{4(3d-16)}{d-5} N_B L^A_C W^{BC \sharp} 
+ \tfrac{4(d-4)}{d-5} N_B N^C (\hd_C {W}^{BA}{}_{\pdot\hh\pdot}    )^\sharp    \Big) \circ \hd^T_{A} \\
&+\tfrac{2d^4 - 21d^3 + 95d^2 - 200d + 152}{(d-1)(d-2)^2 (d-4)^2} K^2 
+ \tfrac{d-2}{(d-4)^2} L\hh \csdot \hh  J 
- \tfrac{2(d-2)}{(d-4)^2} L^4
+ \tfrac{2(2d^3 - 27d^2 + 118d - 172)}{(d-4)^2} L \csdot F \csdot L 
\\
&+ \tfrac{(3d-16)(d-3)^2}{d-4} F^2- \tfrac{2(d-5)(d-6)}{(d-4)^2} L^{AC} L^{BD} {\bar W}_{ABCD} 
+ \tfrac{(d-5)(d-8)}{(d-4) (d-7)}  (\hdb_A L_{BC}) (\hdb^A L^{BC}) \\&
+ 4 W_{N}{}^{B \sharp} \circ  W_{NB}{}^\sharp- 4 N_B L_{AC} \Big(\hd^A W^{BC}{}_{\pdot\hh \pdot} \Big)^\sharp\, .
\end{align*}
\end{theorem}

\medskip 
\noindent
The above result can in fact be profitably used 
in dimensions $d=5,6,7$ by a 
 dimensional continuation argument
because $(I \csdot \hd)^4$ is well-defined in all these dimensions, so by construction so too must be the operator appearing in the theorem.
 The  $d = 5$ case is clearly the most complicated of these continuations, and has been performed in the cases when $\PanE$ acts on
scalars, the normal tractor, and standard tractors given by $\hd^T$ of a scalar density; see Corollary~\ref{forIamamerecorollary}, Theorem~\ref{P4N} and Equation~\nn{throwback}. Dimensions $d=6,7$ are discussed below:

First consider  $d = 6$. While there are no explicit poles here, since the tractor
$$\hd_A W \stackrel g= -2 Y_A W + Z_A^a \nabla_a W - \tfrac{1}{d-6} X_A (\Delta W - 2 J^g W)\, ,$$
it actually has a pole at $d = 6$. However, note that in the theorem above, the tractor $\hd_A W$ only appears when contracted  with $N^A$ or $L^{AB}$. In both cases, the pole can be eliminated by first working in an arbitrary dimension, there noting that $N^A X_A = 0 = L^{AB} X_A$, and then continuing down to six dimensions.

\smallskip

Now consider the case where $d = 7$, which has  one explicit pole with residue $(\hdb L)^2$, and an implicit one:
$$\hdb_A K \stackrel g= -2 \bar{Y}_A K + \bar{Z}_A^a \bn_a K - \tfrac{1}{d-7} X_A (\bar{\Delta} K - 2 \bar{J}^{\bar g} K)\,.$$
However, note that in $d = 7$,
$$(\hdb L)^2 \stackrel g= \tfrac{1}{2} \big(\bar{\Delta}  -2 \bar{J}^{\bar g}\big) K= \tfrac12 \bar \square_Y  K\,,$$
and acting on sections of $\Gamma(\ce M[\tfrac{5-d}{2}])|_{\Sigma}$,
$$(\hdb^A K) \hd^T_A \stackrel g= \tfrac{d-5}{2(d-7)} (\bar{\Delta}  - 2 \bar{J}^{\bar g} )K + \text{regular}\,,$$
where ``regular'' stands for terms that are regular in dimension $d = 7$.
Therefore, away from $d=7$
$$
-\tfrac{2(d^2 - 8d + 17)}{(d-1)(d-2)(d-5)} (\hdb^A K) \hd^T_A + \tfrac{(d-5)(d-8)}{(d-4)(d-7)} (\hdb L)^2 = \tfrac{d^3-7d^2+8d+8}{2(d-1)(d-2)(d-4)} (\bar{\Delta}  - 2 \bar{J})K + \text{regular}\,.
$$
Thus, the operator $\PanE$  can indeed be continued to seven dimensions.

\medskip

Our next task is to prove the central Theorem~\ref{Big_Formaggio}. Note that the algorithm presented below can be used to decompose 
quite general tractor operators into tangential and 
higher transverse order pieces, the latter captured by  powers of the $I\csdot \hd $ operator.

\begin{proof}[Proof of Theorem~\ref{Big_Formaggio}]
As dictated by Theorem~\ref{Pk}, we begin with the operator~$I\csdot \hd^4$, remembering that we will eventually restrict to the hypersurface~$\Sigma$. Also, we are ultimately interested in this operator acting on tractors of weight~$\frac{5-d}2$. However, to begin with we will take the weight to be arbitrary, equal to~$(h-d)/2$, and such that denominators of the form~$h-k$ for certain~$k$  are avoided. Later we will employ a weight and dimension continuation argument. Note that all appearances of the operator~$\hat D$ in what follows are in fact well-defined even when~$h=5$. When problematic poles in the parameter~$h$ appear, we will draw the reader's attention to how these are handled.

Our strategy is to perform a series of manipulations  converting the operator~$I\csdot \hd^4$ to the operator 
~$\hat D^T_A\circ I\csdot \hd^2 \circ \hd^{T\hh A}$ plus terms of lower transverse order. 
The transverse order of an operator 
measures the maximal number of normal derivatives; a detailed definition may be found  in~\cite{BGW}.
The first step is to note that 
~$I\csdot \hd  =\hd_A \circ I^A$ (this follows by contracting the last identity in Lemma~\ref{Dcomm} with the tractor metric and the fact that~$P_A{}^A= 0 = X_A P^{AB}$). Thus
$$
I\csdot \hd^4=\hd_A \circ I^A I\csdot \hd^2\circ I^B  \hd_B\, .
$$
We would like to trade the explicit appearances of scale tractors~$I^A$ and~$I^B$ in the above display for an extension of the tractor first fundamental form~$I^{AB}_{\rm ext}:=h^{AB} - I^A I^B$. 
For that we must first 
 bring~$I^A$ and~$I^B$ together using Equation~\nn{ID,I}, which gives
\begin{eqnarray}
I\csdot \hd^4&=& \phantom{-}\hd_A \circ I\csdot \hd\circ
I^A I^B
\circ
I\csdot \hd
 \circ \hd_B+ {\mathcal R}_1\nonumber
 \\[1mm]
 &=&
  -\hd_A \circ I\csdot \hd\circ
I^{AB}_{\rm ext}
\circ
I\csdot \hd
 \circ \hd_B+ 
{\mathcal R}_2
 \, .\label{theabovedisplay}
\end{eqnarray}
Here 
$ {\mathcal R}_1:=
\hd_A \circ [I^A,I\csdot \hd] \circ
I\csdot \hd
\circ I^B
\hd_B
+
\hd_A \circ I\csdot \hd \circ I^A\circ
[I\csdot \hd
, I^B]\circ
\hd_B
$
has transverse order no more than three---see Equation~\nn{ID,I}. 
The second remainder term~${\mathcal R}_2:={\mathcal R}_1+ 
 \hd_A \circ I\csdot \hd^2
 \circ \hd^A$ also has transverse order no more than three which can be verified using the identity~$\hd_A \hd^A=0$ and 
Proposition~\nn{DD-comm}, which shows that the commutator of a pair of Thomas-$D$ operators has transverse order one.
Later, to simplify~$\mathcal{R}_2$, we will apply Equation~\nn{D2D}.
We will also handle  lower transverse order remainder terms later.

We  employ the identity
$$
I_{\rm ext}^{AD}=
I^{AB}_{\rm ext}h_{BC}I^{CD}_{\rm ext}+ \mathcal{O}(\sigma^d)\, ,
$$
to produce a pair of extensions of the tractor first fundamental form. 
(Also observe  from Equation~\nn{DT} that 
$I_{AB} \hd^B$ is very nearly~$\hd_A^T$.)
We then use 
Equation~\nn{ID,I} and Lemma~\ref{Dcomm}
to rewrite Equation~\nn{theabovedisplay} as
$$
I\csdot \hd^4= -I_{\rm ext}^{AB} \hat D_B \circ 
I\csdot \hd^2
\circ
 I^{\rm ext}_{AC}\hd^C+{
\mathcal R }_3\, .$$
Similar arguments to above show that the latest remainder
${
\mathcal R }_3$ still has transverse order at most three.
Using~Equation~\nn{DT}, note that
$\hd^T_A = I_{AB}^{\rm ext} \hd^B
+\frac{1}{d+2w-3}\, X_A I\csdot \hd^2 
$.
Note that we encounter no poles applying this identity to the above display when~$h=5$. This maneuver produces
\begin{multline}\label{noidea}
I\csdot \hd^4=-\hat D^T_A\circ I\csdot \hd^2 \circ \hd^{T\hh A}
+\tfrac1{h-3}
\hat D^T_A\circ I\csdot \hd^2 \circ X_A I\csdot \hd^2 
\\
+\tfrac1{h-9}X_A I\csdot \hd^4 \circ  \hd^{T\hh A}
-
\tfrac{1}{(h-3)(h-9)}
X_A I\csdot \hd^4  \circ X^A I\csdot \hd^2 
+{\mathcal R }_3\, .
\end{multline}
Applying identities from Lemma~\ref{TIDU}, for the second, third, and fourth terms on the right hand side above, we find 
\begin{align*}
\hat X^A I\csdot \hd^4 \circ D_A^T &=
\tfrac{h(h-7)(h-d-8)}{2(h-3)(h-4)} \, I\csdot \hd^4+\cdots \, ,\\
\hat D^T_A \circ I\csdot \hd^2 \circ X^A I \csdot \hd^2 &= 
\tfrac{(h-5)(h-6)(h+d-10)}{2(h-4)(h-9)}\, I\csdot \hd^4+\cdots\, , \\
X_A I \csdot \hd^4 \circ X^A I \csdot \hd^2 &= -\tfrac{6(h-d-8)}{h-4} I \csdot \hd^4 + \cdots \,,
\end{align*}
where $\cdots$ represents  lower transverse order terms. 
Then, after collecting all the terms in Equation~\nn{noidea} containing~$I\csdot \hd^4$ on the left hand side, we have
$$
\tfrac{(2h-9)(h+d-10)}{(h-3)(h-4)(h-9)}\, 
I\csdot \hd^4=- \hat D^T_A\circ I\csdot \hd^2 \circ \hd^{T\hh A}
+{\mathcal R }_4\, .
$$
When~$h=5$, the coefficient on the left hand side is~$-\frac{d-5}8$. Moreover, at that value of $h$  the operator ~$I\csdot  \hd^2$ in the first term on the right hand side acts on tractors of
weight~$1-\frac{d-1}2$, in which case it is tangential and equals~$P_2^{\Sigma\hookrightarrow M}$.
Therefore we have established that
$$
\PanE  \stackrel\Sigma=  \hh \tfrac{8}{d-5}\hh \big(\hd^{T A}_{\sss\Sigma} \circ
P_2^{\sss\Sigma\hookrightarrow M}
\circ
 \hd^T_{{\sss\Sigma}\hh A} -{\mathcal R}\big)\, .
$$
It remains to compute
the operator~${\mathcal R}$  by evaluating~${\mathcal R}_4$ along~$\Sigma$ in the limit~$h\to 5$. 

 
 We must therefore now discuss the  poles in~${\mathcal R}_4$. The issue is
 that  we cannot use Equation~\nn{DT} to convert  the operator~$\hd$ to~$\hd^T$ when acting on tractors of weight~$1-\frac{d-1}2$.  However, since we know that the limit~$h\to 5$ of the operator~$I\csdot \hd^4$ is well-defined, we may employ Equation~\nn{DT} at general weights and apply a limiting procedure at the end of our calculations.
 
 \medskip

Returning to~${\mathcal R}$, we first   apply Lemma~\ref{D2D} to simplify the term~$\hd_A I \csdot \hd^2 \hd^A$ in~$\mathcal{R}_4$.  To compute~$\mathcal{R}_4$, we employ an
algorithm whose starting point~${\mathcal R}_4$ is an operator of transverse order no more than three,  that acts on an arbitrary weight tractors, and is  evaluated along~$\Sigma$.
Moreover,~${\mathcal R}_4$ is expressed as a sum of words (each of which has transverse order no more than three) composed of operator-valued letters in the alphabet given by 
the scale~$\sigma$, the scale tractor~$I$, the canonical tractor~$X$, the Thomas-$D$ operator, the~$W$ tractor,  the Thomas-$D$ operator acting on any of the other letters (possibly multiple times), and rational functions of~${\sf h}$. Note that 
the tractor identities derived so far can be used to simplify these words, \textit{e.g.},~$X^A \circ (\hd_A I_B) = 0 
$. 
Our algorithm  manipulates  such words and letters. We also introduce the distinguished letter
$$
\hat y:=-I\csdot \hd\, .
$$
The aim of the algorithm is to iteratively convert any word of transverse order~$\ell$ into the form $${\sf Op}\circ {\hat y}^\ell \circ f({\sf h})+\cdots\, ,$$ where the terms~$\cdots$ have transverse order lower than~$\ell$ and $f({\sf h})$ is some rational function of the  operator ${\sf h}$. Here~${\sf Op}$ is some tangential operator that may involve an additional  letter~$\hd^T\!$. Let us first  sketch main ideas of the algorithm:

Step 0 of the algorithm takes any word containing rational functions of~${\sf h}$
 and rewrites those words by shifting these operators to the right end of the word. Then, it applies a simplification-type procedure. While in principle, this simplification is not necessary for the algorithm to achieve the goal desired, it significantly reduces computational complexity in the implementation in the attached documentation~\cite{FormFiles}. This step will be repeated after each of the following steps. 
 
 Step 1 takes any word ending in $f_1({\sf h})$ and rewrites it in the order 
$$U^{\Theta_1}_1 \cdots \circ T^{\Phi_1 \sharp}_1 \cdots \circ \hd^T_{A_1} \cdots \circ \hd_{B_1} \cdots \circ {\hat y}^{\ell} \circ f_2({\sf h})\,,$$
for~$\{U^{\Theta_i}_i\}$ some set of multiplicative letters (acting by tensor multiplication by a tractor),~$\{T^{\Phi_j \sharp}_j\}$ some set of tractor-valued tractor-endo\-morphism letters, and~$f_2({\sf h})$ some possibly new rational function of the letter ${\sf h}$. 

Step 2 prepares to combine pairs of letters~$I^A$ and~$\hd_A$ into~$-{\hat y}$ by commuting the corresponding~$I$'s to the right of the multiplicative letters, the tractor-valued tractor endomorphisms, and the tangential Thomas-$D$ operators. That is, words containing~$I^A$ and~$\hd_A$ are manipulated to take  the form
$$U^{\Theta_1}_1 \cdots \circ T^{\Phi_1 \sharp}_1 \cdots \circ \hd^T_{A_1} \cdots \circ I^{B_k} \circ \hd_{B_1} \cdots \hd_{B_k} \cdots \circ {\hat y}^{\ell} \circ f({\sf h})\,.$$\

Step 3  applies the tractor lemmas above to push every~$I^{B_k}$ past Thomas-$D$ operators with different indices until it is left-adjacent to its corresponding~$\hd_{B_k}$, and then combines these terms to form~$-{\hat y}$. 

Step 4 reapplies Step 1 (so that this newly formed~${\hat y}$ letter is commuted to the right), leaving us with words of the form
$$U^{\Theta_1}_1 \cdots \circ T^{\Phi_1 \sharp}_1 \cdots \circ \hd^T_{A_1} \cdots \circ \hd_{B_1} \cdots \circ {\hat y}^{\ell} \circ f({\sf h})\,,$$
where no letter~$U^{\Theta_i}_i$ is a letter~$I^{B_i}$ with corresponding letter~$\hd_{B_i}$. 

Step 5 rewrites~$\hd_{B_1}$ (which by the previous step has no corresponding~$I^{B_1}$) in terms of the tangential Thomas-$D$ operator,~$\hd^T_{B_1}$. 

Step 6  repeats the previous five steps so long as any letters~$\hd$ remain.
The output of the algorithm is a linear combination of words of the form
$$U^{\Theta_1}_1 \cdots \circ T^{\Phi_1 \sharp}_1 \cdots \circ \hd^T_{A_1} \cdots \circ I \csdot \hd^{\ell} \circ f({\sf h})\,.$$

We now present the algorithm in full detail.

\begin{description}
\item[Step 0] $\phantom{+}$

\begin{description}
\item[Step 0a]
For every letter pair of an operator~$f({\sf h})$ and some tractor-valued operator~$T: \Gamma(\ct^\Phi M[w])$ $\rightarrow$ $\Gamma(\ct^\Theta M[w'])$ with ``operator weight'' $w'-w$, replace~$f({\sf h}) \circ T^\Theta_\Phi$  by~$T^\Theta_\Phi \circ f({\sf h}+2(w'-w))$. 
%

\item[Step 0b]
For any word beginning with multiplicative letters, combine those letters to reduce complexity using the definitions found in~\ref{tractor-defs}, the Leibniz failure~\ref{leib-failure}, and Lemmas~\ref{Dcomm} and~\ref{IDcomm}. For example, one may write~$X_A(\hd^B \hd^C I^{A})=-P^{BC}$ or~$I_A P^{AB}=-\tfrac{1}{d-2} KX^B$.
\end{description}

\item[Step 1]
Repeat the following sub-steps until a full iteration leaves the expression unchanged ({\it i.e.}, ``repeat until termination''):

\begin{description}
\item[Step 1a]
Rewrite all two-letter pairs~$I^A \circ \hd_A$ as~$-{\hat y}$.

\item[Step 1b] Rewrite every instance of the letter combination~$[\hd_A, \hd_B]$ following  Proposition~\ref{DD-comm}.

\item[Step 1c]
Repeat until termination: For every letter of the form~$T^{\Phi \sharp}$ and every multiplicative letter~$U^\Theta \in \Gamma(\ct^\Theta M[w])$, rewrite  pairs~$T^{\Phi \sharp} \circ U^\Theta$ as
$$(T^{\Phi \sharp} U^\Theta) + U^\Theta \circ T^{\Phi \sharp}\,.$$

\item[Step 1d]
Repeat until termination: For some operators~${\sf Op}_1$,~${\sf Op}_2$, and~${\sf Op}_3$, apply Proposition~\ref{leib-failure} and Lemmas~\ref{IDcomm} and~\ref{DWcomm} to write all words of the form~${\sf Op}_1 \circ {\hat y} \circ {\sf Op}_2 \circ {\hat y}^\ell$ as 
$${\sf Op}_1 \circ {\sf Op}_2 \circ {\hat y}^{\ell +1} + {\sf Op}_3 \circ {\hat y}^{\ell}\, .$$
By construction~${\sf Op}_3$ has fewer~${\hat y}$'s (including pairs~$I^A \circ\hd_A$) in it than the operator~${\sf Op}_1 \circ {\sf Op}_2$. Then apply Step 0.

\item[Step 1e]
Repeat until termination: For any operators~${\sf Op}_1$,~${\sf Op}_2$ (where ${\sf Op}_2$ does not contain the letter $\hd$), and~${\sf Op}_3$, apply Proposition~\ref{leib-failure} and Lemmas~\ref{Dcomm} and~\ref{DWcomm} to write all words of the form~${\sf Op}_1 \circ \hd^{A_1} \circ {\sf Op}_2 \circ \hd^{A_2} \cdots \hd^{A_k} \circ {\hat y}^{\ell}$ as
$${\sf Op}_1 \circ {\sf Op}_2 \circ \hd^{A_1}\cdots \hd^{A_k} \circ {\hat y}^{\ell} + {\sf Op}_3 \circ \hd^{A_2} \cdots \hd^{A_k} \circ {\hat y}^{\ell}  \,.$$
By construction~${\sf Op}_3$ has no more operators~$\hd$ in it than the operator~${\sf Op}_1$. Then apply Step 0.

\item[Step 1f]
Repeat until termination: Use Proposition~\ref{leib-failure} and Lemmas~\ref{Dcomm} and~\ref{IDcomm}, for every combination of letters~$\hd^T \circ I$ to the left of every appearance of $\hd$ or ${\hat y}$, to faithfully replace it with
$$I \circ \hd^T + {\sf Op}_1^T \,.$$
By construction the operator~${\sf Op}^T_1$ is  tangential. Similarly, for every combination of letters~$\hd^T \circ X$, replace it with
$$X \circ \hd^T + {\sf Op}_2^T \, .$$
Again the  operator~${\sf Op}^T_2$ must be  tangential. Apply similar identities to letters of the form $\hd^T \circ U^\Phi$ where $U^\Phi$ is a multiplicative letter. Then apply Step 0.
%
\end{description}

\item[Step 2] Repeat until every word containing at least one pair~$I^A \circ {\sf Op} \circ \hd_A$ is written as~${\sf Op}' \circ I^{A_k} \circ \hd_{A_1} \cdots \hd_{A_k}$, where~${\sf Op}$ and~${\sf Op}'$ are some combination of letters:
\begin{description}
\item[Step 2a] For some operators~${\sf Op}_1$ and~${\sf Op}_2$,  and for every multiplicative letter~$U^\Theta \in \Gamma(\ct^\Theta M[w])$, rewrite every word of the form~${\sf Op}_1 \circ I^A \circ U^\Phi \circ {\sf Op}_2$ as
$${\sf Op}_1 \circ U^\Phi \circ I^A \circ {\sf Op}_2\,.$$

\item[Step 2b]  For some operators~${\sf Op}_1$,~${\sf Op}_2$, and~${\sf Op}_3$, and for each tractor-valued tractor-endomorphism letter~$T^{\Phi \sharp}$, rewrite every word of the form~${\sf Op}_1 \circ I^A \circ T^{\Phi \sharp} \circ {\sf Op}_2 \circ \hd_A \circ {\sf Op}_3$ as
$${\sf Op}_1 \circ T^{\Phi \sharp} \circ I^A \circ {\sf Op}_2 \circ \hd_A \circ {\sf Op}_3 - {\sf Op}_1 \circ T^{\Phi A}{}_B I^B \circ {\sf Op}_2 \circ \hd_A \circ {\sf Op}_3\,.$$

\item[Step 2c] For some operators~${\sf Op}_1$,~${\sf Op}_2$, and~${\sf Op}_3$, apply Lemmas~\ref{Dcomm} and~\ref{IDcomm} as well as the definition of~$\hd^T$ in Definition~\ref{tractor-defs} to rewrite every word of the form~${\sf Op}_1 \circ I^A \circ \hd^T \circ {\sf Op}_2 \circ \hd_A \circ {\sf Op}_3$ as
$${\sf Op}_1 \circ \hd^T \circ I^A \circ {\sf Op}_2 \circ \hd_A \circ {\sf Op}_3 + {\sf Op}'_A \circ {\sf Op}_2 \circ \hd^A \circ {\sf Op}_3\,,$$
for some~${\sf Op}'_A$ that does not contain the letter~$I^A$. 
This step is designed to only move $I$'s to the right when they can be contracted on to $\hd$'s.
Apply Step 0.

\end{description}

\item[Step 3]
Repeat until each word has one fewer (or zero) pair(s) of letters~$I^{A_i}$ and~$\hd_{A_i}$. For some operators~${\sf Op}_1$ and~${\sf Op}_2$, apply Proposition~\ref{leib-failure} and Lemma~\ref{Dcomm} to rewrite every word of the form~${\sf Op}_1 \circ I_{A_i}\circ \hd^{A_1} \cdots \hd^{A_i} \cdots \hd^{A_k} \circ {\hat y}^{\ell}$ as
$$-{\sf Op}_1 \circ  \hd^{A_1} \cdots {\hat y} \cdots \hd^{A_k} \circ {\hat y}^{\ell} + {\sf Op}_2 
 \circ {\hat y}^{\ell}\,,$$
where~${\sf Op}_2$ is some operator that does not contain~$I_{A_i}$. Apply Step 0.

\item[Step 4] Reapply Step 1.

\item[Step 5] In any given word, by virtue of Proposition~\ref{Dt-trans}, rewrite the left-most letter~$\hd_A$ as~$\hd^T_A -I_A\circ {\hat y} - X_A \circ {\hat y}^2\circ \tfrac{1}{h-3}$.

\item[Step 6] If any word contains the letter~$\hd$, repeat Steps 0 through 5.
\end{description}

The remainder of the calculation amounts to rewriting combinations of non-derivative letters
$$U_1^{\Theta_1} \cdots \circ T_1^{\Phi_1 \sharp} \cdots$$
in terms of hypersurface tractors via holographic formul\ae. This process is generally tedious and relies on dozens of identities arising from the Leibniz failure  (both on the hypersurface and in the ambient space), Lemma~\ref{Dcomm} (and its direct application to the hypersurface operator~$\hdb$), Lemma~\ref{IDcomm}, and the Gauss-Thomas  formula~\nn{thesecondhorcrux}. Some of these identities can be found in~\cite{BGW}. Finally, we can take the limit~$h \rightarrow 5$ to resolve~$\mathcal{R}$. This entire procedure 
was implemented 
using FORM in 
the file 

\smallskip
\begin{center}
\texttt{anc/General-tensor/Paneitz-tensor-algorithm2.frm}. 
\end{center}
\smallskip

\noindent
That the proposed algorithm here terminates requires
no proof, since we explicitly verified that six iterations suffices.
The result obtained from this computation is below:
\begin{align*}
-\tfrac{8}{d-5} {\mathcal R}
  =
&+\Big(\tfrac{4(2d-11)}{d-5} L^{AC} L_C^B  
+ \tfrac{4(d-3)(d-6)}{d-5} F^{AB} 
- \tfrac{8}{d-5} W^{AB\, \sharp} \Big) \circ \hd^T_{A}\circ \hd^T_{B} \\
&+\Big( \tfrac{4(2d-11)}{d-5} L_{BC} (\hdb^B L^{CA}) 
- \tfrac{2(d^2 - 8d+17)}{(d-1)(d-2)(d-5)} (\hdb^A K) 
+ \tfrac{4(3d-16)}{d-5} N_B L_{CD} W^{BCDA} \\&\phantom{+\Big(Q}
- \tfrac{4(3d-16)}{d-5} N_B L^A_C W^{BC \sharp} 
+ \tfrac{4(d-4)}{d-5} N_B N^C (\hd_C {W}^{BA}{}_{\pdot\hh\pdot}    )^\sharp    \Big) \circ \hd^T_{A} \\
&+\tfrac{2d^4 - 21d^3 + 95d^2 - 200d + 152}{(d-1)(d-2)^2 (d-4)^2} K^2 
+ \tfrac{d-2}{(d-4)^2} L\hh \csdot \hh  J 
- \tfrac{2(d-2)}{(d-4)^2} L^4
+ \tfrac{2(2d^3 - 27d^2 + 118d - 172)}{(d-4)^2} L \csdot F \csdot L 
\\
&+ \tfrac{(3d-16)(d-3)^2}{d-4} F^2- \tfrac{2(d-5)(d-6)}{(d-4)^2} L^{AC} L^{BD} {\bar W}_{ABCD} 
+ \tfrac{(d-5)(d-8)}{(d-4) (d-7)}  (\hdb_A L_{BC}) (\hdb^A L^{BC}) \\&
+ 4 W_{N}{}^{B \sharp} \circ  W_{NB}{}^\sharp- 4 N_B L_{AC} \Big(\hd^A W^{BC}{}_{\pdot\hh \pdot} \Big)^\sharp\, .
\end{align*}
The above expression for $\mathcal{R}$ contains operators of the type $\mathcal{W}^A \circ \hd^T_A \circ \hd^T_B$ which, strictly, are not defined. However, because $\mathcal{W}^A \circ X_A = 0$, as discussed earlier, we have used 
 use the notation~$\hd^T$
 for the left-most Thomas-$D$ operator. This completes the proof.
\end{proof}

\subsection{Proofs of Results~\ref{forIamamerecorollary}, \ref{P4N},\ref{dots}, \ref{Voldemort}, and~\ref{Ronkonkoma}}

\begin{proof}[Proof of Corollary~\ref{forIamamerecorollary}]
The proof mainly amounts to an application of Theorem~\ref{Big_Formaggio}. 
Because the operator acts on scalar densities, we may use Theorem~\ref{thesecondhorcrux} to convert operators~$\hd^T$ to~$\hdb$ plus lower order terms. 
 The proof then splits into two separate computations. The first expresses~$\hdb^A \circ P_2^{\sss\Sigma\hookrightarrow M} \circ \hdb_A$ 
 in terms of Riemannian operators, while the second similarly 
 handles the subleading terms. The entire computation is carried out in FORM: the first computation can be found in the file
\begin{center} 
 \texttt{anc/Paneitz-scalar/DbID2Db-scalar.frm}
 \end{center}
and the second  in
\begin{center} \texttt{anc/Paneitz-scalar/Paneitz-scalar-Riemannian.frm}.
\end{center}
The final step uses Equation~\nn{fourth} to  rewrite $\IVo_{ab}$ in terms of $C_{\hat n (ab)}^\top$ in order that the result can be continued to $d = 5$ (and subsequently used in Theorem~\ref{BQC}). Well-definedness of the final result in $d=5$ can be established by inspection.  
\end{proof}

\begin{proof}[Proof of Theorem~\ref{P4N}]
As in Corollary~\ref{forIamamerecorollary}, this theorem is also an application of Theorem~\ref{Big_Formaggio}. 
Note that the normal tractor has weight zero in all dimensions, while the $\PanE$ acts on  tractors
of weight $\frac{5-d}{2}$. Moreover, to use Theorem~\ref{Big_Formaggio} in five dimensions we must compute in general~$d$ and then continue to $d=5$. Thus 
 we introduce a scalar density $\tau$ of weight $\frac{5-d}{2}$ and instead compute $\PanE (N^A \tau)$, continue to $d=5$, and thereafter set $\tau$ to unity. Similarly to the previous proof, handling the term  $\hd^{T \, B} \circ P_2^{\sss\Sigma\hookrightarrow M} \circ \hd^T_B (N^A \tau)$ term is challenging. It is computed in  the FORM file 
\begin{center} 
 \texttt{anc/Paneitz-N/DtID2Dt-N.frm}
 \end{center}
 The remainder of the computation is performed by the file
  \begin{center} 
 \texttt{anc/Paneitz-N/Paneitz-N-Riemannian.frm}
  \end{center}
  Finally, we again write $\IVo_{ab}$ in terms of $C_{\hat n (ab)}^\top$ so that the result is well-defined in $d = 5$.
\end{proof}

\begin{proof}[Proof of Lemma~\ref{dots}]
This calculation is a significant exercise in Riemannian hypersurface geometry. We use three facts: First, acting on a weight $w\neq 1-\frac d2$ tractor, the operator $I \csdot \hd$ is given by
$$I \csdot \hd \stackrel g= \nabla_n + w \rho - \tfrac{s}{d+2w-2} (\Delta + w J)\,.$$
Second, written in terms of the canonical extension $\IIo^{\rm e}_{ab}  = Z^A_a Z^B_b P_{AB}$, we have that $K^{\rm e} = (\IIo^{\rm e})^2$. Third, this canonical extension can be written as $\IIo^{\rm e}_{ab} = \nabla_a n_b + s P_{ab} + \rho g_{ab}$ (see the discussion in~\cite{BGW}). Using these facts, we can recast  $\dddot{K} = I \csdot \hd^3 K^{\rm e}$ as a Riemannian operator on tensors, at which point the problem reduces to standard (albeit lengthy) hypersurface calculations, carried out in the FORM program
\begin{equation*}
\texttt{anc/Riemannian-identities/Kddd.frm}
\end{equation*}

In order to obtain a manifestly invariant result, we note that the operator $L^{ab}$ given by
$$\Gamma(\odot^2_\circ T^*\Sigma[4-d]) \ni t \mapsto
\bn\csdot \bn\csdot t + \bar  P\csdot t \in
 \Gamma(\ce \Sigma[-d])\, ,$$
 described in~\cite{hypersurface_old} can act on three independent structures appearing  in $\dddot{K}$: namely,
$$\IIo^3_{(ab)\circ},\; 2 \IIo_{(a} \csdot \Fo_{b)\circ},\; K \IIo_{ab} \in \Gamma(\odot^2_\circ T^*\Sigma[4-d])\,.$$
The conversion (in five dimensions) of the result to expressions involving the weight $-1$ densities $L^{ab} \IIo^3_{(ab)\circ}$, $2 L^{ab} (\IIo_{(a} \csdot \Fo_{b)\circ})$, and $L^{ab} (K \IIo_{ab})$ 
is also carried out in the above FORM file. The tractor expression $W_{ABCN}^\top \hdb^A F^{BC}$ can be computed with standard tractor techniques and is also included in the above FORM computation.
\end{proof}

\begin{proof}[Proof of Theorem~\ref{Voldemort}]
In the proof of Theorem~\ref{P4N} above, we showed how to calculate $\PanE N^A$. As noted in Equation~\nn{FORMAGGINO}, to calculate the obstruction density ${\mathcal B}_\Sigma$, we should first compute $\PanE N^A - 3 I \csdot \hd^3 (X^A K_{\rm e})$ and in turn
$$\tfrac{1}{1440}  (\bar{D}_A \circ \top) \big(\PanE N^A - 3 I \csdot \hd^3 (X^A K_{\rm e}) \big) \,.$$
To do so, we start with $I \csdot \hd^3 (X^A K_{\rm e})$, and employ a combination of the  tractor and hypersurface calculi developed above. This is carried out in the FORM program
\begin{equation*}\texttt{anc/Paneitz-N/ID3xK.frm}\end{equation*} 
Note that this computation involves $\dddot{K}$ and thus requires Lemma~\ref{dots}. Combining this result with that for $\PanE N^A$, we can directly evaluate the obstruction; this is carried out in the FORM program
\begin{equation*}
\texttt{anc/Paneitz-N/Obstruction-d5.frm}
\end{equation*}
\end{proof}

\begin{proof}[Proof of Theorem~\ref{Ronkonkoma}]
Because $\Sigma$ is embedded as the conformal infinity of a Poincar\'e--Einstein structure $(M,\cc)$, from Theorem 1.8 of~\cite{BGW}, we have that $\IIo_{ab} = \Fo_{ab} = C_{\hat n (ab)}^\top = 0$. Consequently, from Theorem 1.3 of the same, $\hd^T \eqSig \hdb$. Note that another straightforward consequence is that $C_{a \hat n \hat n} \eqSig B_{\hat n a} \eqSig 0$ (see~\cite[Proposition 4.3]{Goal}), which implies that $W_{NABC} \eqSig 0$. Thus, to prove the theorem, from Equation~\nn{throwback}, it  suffices to compute $\hdb^A \circ \PanE \circ \hdb_A$. According to Theorem~\ref{Big_Formaggio}, in this case we have that 
\begin{align*}
\PanE \eqSig \tfrac{8}{d-5} \left(\hdb^A \circ P_2^{\sss\Sigma\hookrightarrow M} \circ \hdb_A - W^{AB \hash} \circ \hdb_A \circ \hdb_B + \tfrac{d-4}{2} N_B N^C (\hd_C W^{BA}{}_{\pdot \hh \pdot})^\hash \circ \hdb_A \right)\, .
\end{align*}
We  explicitly compute the operator~$\hdb^A \circ \PanE \circ \hdb_A$ using FORM to computing in dimension~$d$, and then continue to $d =5$. First, we compute
$$(\hdb^B \circ  P_2^{\sss\Sigma\hookrightarrow M} \circ \hdb_B \circ \hdb_A) f$$
for $f \in \Gamma(\ce \Sigma[\tfrac{7-d}{2}])$ using the FORM file
\begin{equation*}
\texttt{anc/P6/DbP2Db-Dbf.frm}
\end{equation*}
We take the resulting expression and feed it into the following file which completes the calculation:
$$
\texttt{anc/P6/P6-PE.frm} 
$$
This outputs the result displayed in the theorem.
\end{proof}

%
%
%
%
%
%

%
%
%

\section*{Acknowledgements}

A.W.~was also supported by a Simons Foundation Collaboration Grants for Mathematicians ID 317562 and 686131, and  thanks the University of Auckland for warm hospitality.
A.W. and A.R.G.
 gratefully acknowledge support from the Royal Society of New Zealand via Marsden Grants 16-UOA-051 and 19-UOA-008.

\end{document}